\documentclass[reqno, 12pt]{amsart}
\usepackage{amsmath,amsxtra,amssymb,latexsym, amscd,amsthm}
\usepackage{graphicx,color}
\usepackage[utf8]{inputenc}
\usepackage[mathscr]{euscript}
\usepackage{mathrsfs}
\usepackage[english]{babel}
\usepackage{enumerate}
\usepackage{hyperref}
\newtheorem {theorem}{Theorem}[section]
\newtheorem {corollary}[theorem]{Corollary}
\newtheorem {lemma}[theorem]{Lemma}

\newtheorem {definition}[theorem]{Definition}
\newtheorem {remark}[theorem]{Remark}

\newcommand{\ord}{{\rm ord\ }}
\newcommand{\mult}{{\rm mult }}

\def\ar{a\kern-.370em\raise.16ex\hbox{\char95\kern-0.53ex\char'47}\kern.05em}
\def\ees{{\accent"5E e}\kern-.385em\raise.2ex\hbox{\char'23}\kern-.08em}
\def\eex{{\accent"5E e}\kern-.470em\raise.3ex\hbox{\char'176}}
\def\AR{A\kern-.46em\raise.80ex\hbox{\char95\kern-0.53ex\char'47}\kern.13em}
\def\EES{{\accent"5E E}\kern-.5em\raise.8ex\hbox{\char'23 }}
\def\EEX{{\accent"5E E}\kern-.60em\raise.9ex\hbox{\char'176}\kern.1em}
\def\ow{o\kern-.42em\raise.82ex\hbox{
  \vrule width .12em height .0ex depth .075ex \kern-0.16em \char'56}\kern-.07em}
\def\OW{O\kern-.460em\raise1.36ex\hbox{
\vrule width .13em height .0ex depth .075ex \kern-0.16em \char'56}\kern-.07em}
\def\uw{u\kern-.42em\raise.82ex\hbox{
  \vrule width .12em height .0ex depth .075ex \kern-0.16em \char'56}\kern-.07em}
\def\UW{U\kern-.42em\raise1.55ex\hbox{
\vrule width .13em height .0ex depth .075ex \kern-0.16em \char'56}\kern-.07em}
\def\DD{D\kern-.7em\raise0.4ex\hbox{\char '55}\kern.33em}
\def\OOH{{\accent"5E O}\kern-.78em\raise.8ex\hbox{\char'22}\kern.28em}
\def\UWS{\' \UW }

\title{Limits of real bivariate rational functions}

\author{S\~i Ti\d{\^e}p \DD inh$^1$}
\address{Institute of Mathematics, VAST, 18, Hoang Quoc Viet Road, Cau Giay District 10307, Hanoi, Vietnam}
\email{dstiep@math.ac.vn}

\author{Feng Guo$^2$}
\address{School of Mathematical Sciences, Dalian  University of Technology, Dalian, 116024, China}
\email{fguo@dlut.edu.cn}

\author{H\OOH NG \DD\UWS C NGUY\EEX N}
\address{TIMAS, Thang Long University, Nghiem Xuan Yem, Hanoi, Vietnam}
\address{Institute of Mathematics, VAST, 18, Hoang Quoc Viet Road, Cau Giay District 10307, Hanoi, Vietnam}
\email{nhduc82@gmail.com}

\author{TI\EES N-S\OW N PH\d{A}M$^4$}
\address{Department of Mathematics, Dalat University, 1 Phu Dong Thien Vuong, Dalat, Vietnam}
\email{sonpt@dlu.edu.vn}

\thanks{$^{1, 3, 4}$These authors are funded by International Centre for Research and Postgraduate Training in Mathematics (ICRTM) under grant number ICRTM04$\_$2021.04}
\thanks{$^{2}$Feng Guo was supported by the Chinese National Natural Science Foundation under grant 11571350, the Fundamental Research
Funds for the Central Universities.}

\date{\today}

\begin{document}

\begin{abstract}
Given two nonzero polynomials $f, g \in\mathbb R[x,y]$ and a point $(a, b) \in \mathbb{R}^2,$ we give some necessary and sufficient conditions for the existence of the limit $\displaystyle \lim_{(x, y) \to (a, b)} \frac{f(x, y)}{g(x, y)}.$ 
We also show that, if the denominator $g$ has an isolated zero at the given point $(a, b),$ then the set of possible limits of $\displaystyle  \lim_{(x, y) \to (a, b)} \frac{f(x, y)}{g(x, y)}$ is a closed interval in $\overline{\mathbb{R}}$ and can be explicitly determined.  As an application, we propose an effective algorithm to verify the existence of the limit and compute the limit (if it exists). Our approach is geometric and is based on Puiseux expansions.
\end{abstract}
\maketitle

\section{Introduction}

Computing limits of (real) multivariate functions at given points is one of the basic problems in computational calculus. Unlike the univariate case which has been well studied \cite{Gruntz1996,Salvy1999}, computing limits of multivariate rational functions is a nontrivial problem and is an active research area even in the bivariate case, see \cite{Alvandi2016,Cadavid2013,Strzebonski2018,Strzebonski2021,Velez2017,Xiao2014,Xiao2015,Zeng2020}.

Cadavid, Molina and V{\'e}lez~\cite{Cadavid2013} proposed an algorithm, now available in {\scshape Maple} as the {\sf limit/multi} command,
for determining the existence and possible value of limits of bivariate rational functions. This approach is extended by V{\'e}lez, Hern{\'a}ndez and Cadavid in \cite{Velez2017} to rational functions in three variables and by Alvandi, Kazemi and Maza \cite{Alvandi2016} to multivariate rational functions. These three results  require the given point to be an isolated zero of the denominator and rely on the key observation that, for determining the existence of limits of rational functions, it suffices to study limits along the so--called {\em discriminant variety.}
Also note that, with the aid of the theory of real valuations, Xiao, X.~Zeng and G.~Zeng~\cite[Theorem~3.3]{Xiao2015} obtained necessary and sufficient conditions for the non existence of limits of multivariate rational functions by testing on some polynomial curves; see also Remark~\ref{Remark417} below.

Xiao and Zeng~\cite{Xiao2014} presented an algorithm for determining the limits of multivariate rational functions without any assumption on both the number of variables and the denominators. Their method phrases the limiting problem as a quantifier elimination problem and solves it using triangular decomposition of algebraic systems, rational univariate representation and adjoining infinitesimal elements to the base field. Recently, this approach is improved by Zeng and Xiao~\cite{Zeng2020}, who provided an algorithm for determining the limits of real rational functions in two variables, based on Sturm's theorem and the general Sturm--Tarski theorem for counting certain roots of univariate polynomials in a real closed field.

In \cite{Strzebonski2018}, Strzebo{\'n}ski presented five methods for computing limits of multivariate rational functions by writing the problem as
different quantified polynomial formulations. The methods, which are based on the cylindrical algebraic decomposition algorithm, do not require any assumptions about the rational function and compute the lower limit and the upper limit. Very recently, Strzebo{\'n}ski \cite{Strzebonski2021} 
showed algorithmically that, if the denominator has an isolated zero at a given point, then the problem of computing limits of quotients of real analytic functions can be reduced to the case of rational functions.

This paper addresses to the problem of computing limits of real bivariate rational functions using Puiseux expansions. 
Precisely, let $f, g \in\mathbb R[x,y]$ be nonzero polynomials and $(a, b) \in \mathbb{R}^2.$ We are interested in the problem of deciding whether the following limit exists or not:
\begin{equation*}
\lim_{(x, y) \to (a, b)}\frac{f(x, y)}{g(x, y)}.
\end{equation*}

It is well known that the limit, if it exists, is easily computed (see, for example Lemma~\ref{Lemma41} below). Furthermore, the limit exists and equals to $L \in \mathbb{R}$ if and only if $\mathcal{R}_{f, g} = \{L\},$ where $\mathcal{R}_{f, g}$ stands for the {\em (numerical) range} of limits, which is defined as follows:
$$\mathcal{R}_{f, g} :=\left\{
L \in \overline{\mathbb{R}} : \ \text{there is a sequence } (x_k, y_k)\to(a, b) \ \text{ s.t. } \displaystyle\lim_{(x_k, y_k) \to (a, b)}\frac{f(x_k, y_k)}{g(x_k, y_k)}=L \right\}.$$

We shall present some necessary and sufficient conditions for the existence of the limit. We also shown that if the denominator $g$ has an isolated zero at the point $(a, b),$ then the range $\mathcal{R}_{f, g}$ is a closed interval in $\overline{\mathbb{R}}$ and can be explicitly determined. As an application, we propose an effective algorithm to verify the existence of the limit and compute the limit (if it exists) in the case where $f$ and $g$ are polynomials with rational coefficients. Instead of rational numbers we could use any computable subfield ${K}$ of $\mathbb{R},$ as long as polynomials with coefficients in $K$ are accepted by the Newton--Puiseux algorithm.


Similarly to Cadavid, Molina and V{\'e}lez~\cite{Cadavid2013}, our approach is geometric and is based on Puiseux expansions. Indeed, in order to verify the existence of the limit, it suffices to study, locally around $(a, b),$ the behavior of $f$ along the half-branches of the curve $g = 0$  as well as the behavior of $f$ and $g$ along the half-branches of the {\em discriminant curve} (also known as {\em tangency curve}, see \cite{HaHV2017}) $G_{f, g}  = 0,$ where
\begin{eqnarray*}
G_{f, g} 
&:=& (y - b) \left(g\frac{\partial f}{\partial x}- f\frac{\partial g}{\partial x}\right) - (x - a)\left(g\frac{\partial f}{\partial y} - f\frac{\partial g}{\partial y}\right).
\end{eqnarray*}

In the case where $g$ has an isolated zero at $(a, b),$ the range $\mathcal{R}_{f, g}$ is computed and so the limit (if it exists) is determined by checking, locally around $(a, b),$ the behavior of $f$ and $g$ along the half-branches of the {\em critical curve} $F_{f, g}  = 0,$ where
$$F_{f, g} :=\frac{\partial f}{\partial x}\frac{\partial g}{\partial y}-\frac{\partial f}{\partial y}\frac{\partial g}{\partial x}.$$

Observe that, in general, $\deg F_{f, g} < \deg G_{f, g}$ and the number of monomials of $F_{f, g}$ is smaller than that of $G_{f, g},$
so it seems that our algorithm is simpler and more efficient than that of Cadavid, Molina and V{\'e}lez~\cite{Cadavid2013}.

Canceling common factors of $f$ and $g$ does not affect the limit, hence without loss of generality we may assume that $f$ and $g$ 
have no non-constant common divisor in $\mathbb{R}[x, y].$ Under this assumption, algorithms for determining the limit are proposed by Strzebo{\'n}ski \cite{Strzebonski2018} and by Zeng and Xiao \cite{Zeng2020}. Note that, as shown in Lemma~\ref{Lemma41} below, if the limit exists, then $g$ must have an isolated zero at $(a, b).$ 

The rest of this paper is organized as follows. Section~\ref{SectionPreliminary} contains some preliminaries from semi-algebraic geometry which will be used later. On one hand, we will use Puiseux expansions to study the local behavior of functions along half-branches of curves,
on the other hand, we would like that our algorithm could be implemented in a computer algebra system, so it is imperative to work with truncated Newton--Puiseux roots. To do this, in Section~\ref{Section3}, we make sure that the conditions guaranteeing the existence of of limits of bivariate rational functions can be expressed by using truncated Newton--Puiseux roots. In Section~\ref{Section4}, some necessary and sufficient conditions for the existence of limits of bivariate rational functions is given. A formula for ranges and the main algorithm of the paper are provided in Section~\ref{Section5}. We also implement our algorithm in {\scshape Maple} and examine some examples. Note that, in Example~(20), the {\sf limit} command in {\scshape Maple} 2021 returns a wrong range.
Furthermore, our algorithm can do better than that of {\scshape Maple} 2021 in computing ranges (see Examples~(18)--(21)).
Computing ranges in the case of non isolated zero of the denominator would be object of our future work.

\section{Preliminaries} \label{SectionPreliminary}
Throughout this work we shall consider the Euclidean vector space ${\mathbb R}^n$ endowed with its canonical scalar product $\langle \cdot, \cdot \rangle,$ and we shall denote its associated norm by $\| \cdot \|.$  The closed ball (resp., the sphere) centered at ${x} \in \mathbb{R}^n$ of radius $r$ will be denoted by $\mathbb{B}_{r}({x})$ (resp., $\mathbb{S}_{r}({x})$). When ${x}$ is the origin $0$ of $\mathbb{R}^n$ we write $\mathbb{B}_{r}$ and $\mathbb{S}_{r}$ instead of $\mathbb{B}_{r}(0)$ and $\mathbb{S}_{r}(0)$ respectively.  Set $\overline{\mathbb R}:=\mathbb R\cup\{\pm\infty\}.$

Now, we recall some notions and results of semi-algebraic geometry, which can be found in \cite{Bochnak1998, HaHV2017, Dries1996}.

\begin{definition}{\rm
A subset $S$ of $\mathbb{R}^n$ is called {\em semi-algebraic} if it is a finite union of sets of the form
$$\{x \in \mathbb{R}^n \ | \  f_i(x) = 0, \ i = 1, \ldots, k;\ f_i(x) > 0, \ i = k + 1, \ldots, p\},$$
where all $f_{i}$ are polynomials.
In other words, $S$ is a union of finitely many sets, each defined by finitely many polynomial equalities and inequalities.
A function $f \colon S \rightarrow {\mathbb{R}}$ is said to be {\em semi-algebraic} if its graph
\begin{eqnarray*}
\{ (x, y) \in S \times \mathbb{R} \ | \ y =  f(x) \}
\end{eqnarray*}
is a semi-algebraic set.
}\end{definition}

A major fact concerning the class of semi-algebraic sets is its stability under linear projections.

\begin{theorem}[Tarski--Seidenberg theorem] \label{TarskiSeidenbergTheorem}
The image of any semi-algebraic set $S \subset \mathbb{R}^n$ under a projection to any linear subspace of $\mathbb{R}^n$ is a semi-algebraic set.
\end{theorem}

The following well-known results will be of great importance for us.

\begin{lemma} [monotonicity lemma] \label{MonotonicityLemma}
Let $f \colon (a, b) \rightarrow \mathbb{R}$ be a semi-algebraic function. Then there are finitely many points $a = t_0 < t_1 < \cdots < t_k = b$ such that the restriction of $f$ to each interval $(t_i, t_{i + 1})$ is analytic, and either constant, or strictly increasing or strictly decreasing.
\end{lemma}

\begin{lemma}[semi-algebraic choice] \label{LemmaSAChoice}
Let $S$ be a semi-algebraic subset of $\mathbb{R}^m \times \mathbb{R}^n.$ Denote by $\pi \colon \mathbb{R}^m \times \mathbb{R}^n \to \mathbb{R}^m$ the projection on the first $m$ coordinates. Then there is a semi-algebraic map $f \colon \pi(S) \to \mathbb{R}^n$ such that $(t, f(t)) \in S$ for all $t \in \pi(S).$ 
\end{lemma}

\section{Truncated Newton--Puiseux roots} \label{Section3}

In this section we define truncated Newton--Puiseux roots and establish some preliminary results that will be used in the sequel.
Let $\mathbb{C}\{x, y\}$ and $\mathbb{R}[x, y],$ respectively, stand for the {\em ring of convergent power series} and {\em the ring of real polynomials} in variables $x$ and $y.$

Let $f \in \mathbb{C}\{x, y\}.$ If $f \not\equiv 0,$ we can write
$$f = f_m + f_{m+1} + \cdots,$$
where $f_m\not\equiv 0$ and for each $k \ge m$, either $f_k \equiv 0$ or $f_k$ is a homogeneous polynomial of degree $k;$ then the {\em order} of $f$ is defined by $\ord f \ :=\ m.$ If $f \equiv 0,$ it is convenient to set $\ord f := +\infty.$ 

We say that $f$ is {\em $y$-regular (of order $m$)} if $\ord f = m$ and $f_m(0, 1) \ne 0.$ 
By a linear change of coordinates, we can transform $f$ into a $y$-regular function of the same order as follows.

\begin{lemma}\label{Lemma31}
Let $c \in \mathbb{C}.$ Then the function $\widetilde{f} (x, y) := f(x + cy, y)$ is 
$y$-regular of order $m$ if and only if $f_m(c, 1) \ne 0.$
\end{lemma}
\begin{proof}
This is a direct consequence of the fact that the homogeneous polynomial of the lowest degree of $\widetilde{f}$ is $f_m(x + cy, y).$
\end{proof}

Let $\varphi\not\equiv 0$ be a {\em Puiseux series} of the following form
$$\varphi(x) = c_1 x^{\frac{n_1}{n}} + c_2 x^{\frac{n_2}{n}} + \cdots,$$
where $c_k \in \mathbb{C}$ and $n \leqslant n_1 < n_2 < \cdots$ are positive integers, having no common divisor, such that the power series $\varphi(t^n)$ has positive radius of convergence. 
The {\em order} of $\varphi$ is defined by 
$$\ord \varphi \ :=\ \frac{n_1}{n}.$$
If $\varphi\equiv 0,$ set 
$\ord \varphi := +\infty.$

The series $\varphi$ is {\em real} if all $c_k$ are real. 
Define for $N > 0,$
\begin{eqnarray*}
\varphi \mod x^N &:=& \sum_{\frac{n_k}{n} < N} c_k x^{\frac{n_k}{n}},
\end{eqnarray*}
which is the sum of the terms of $\varphi$ of exponents strictly less than $N.$ 
The series $\varphi$ is called a {\em Newton--Puiseux root} of $f$ if $f(x, \varphi(x)) \equiv 0.$ 

Assume that $f$ is $y$-regular and let
\begin{eqnarray*}
\mathbf{P}(f) &:=& \{\varphi:\ \varphi \text{ is a Newton--Puiseux root of } f\}.
\end{eqnarray*}
In view of Weierstrass's preparation theorem and Puiseux's theorem (see, for example, \cite[The Puiseux theorem, page~170]{Lojasiewicz1991} or \cite[page~98]{Walker1950}), we can write
$$f=u(x,y)\prod_{\varphi\in \mathbf{P}(f)}(y-\varphi(x))^{\mult_f(\varphi)},$$
where $u \in \mathbb{C}\{x, y\}$ with $u(0, 0) \ne 0$ and $\mult_f(\varphi)$ is the {\em multiplicity} of $\varphi.$
In order to provide criteria for the existence of limits of bivariate rational functions, 
some preparations on truncated Newton--Puiseux roots are needed. 

For an integer $N > 0,$ the {\em set of Newton--Puiseux roots of $f$ truncated from $N$} is defined by
\begin{eqnarray*}
\mathbf{P}_N(f) &:=& \{\varphi \mod x^N :\ \varphi \in \mathbf{P}(f)\}.
\end{eqnarray*}

According to the latest version of Maple, the software can only compute (truncated) Puiseux roots of bivariate square-free polynomials.
(Recall that a polynomial is {\em square-free} if it does not have multiple factors.) The following theorem provides a method to compute 
truncated Newton--Puiseux roots of polynomials not necessarily square-free.

\begin{theorem}\label{Theorem32}
Let $f \in \mathbb{R}[x, y]$ be $y$-regular of order $m > 0$ and let $N$ be an integer such that $m N > \deg f.$ Then for any $c \ne 0,$ the polynomial
$$\widetilde f(x, y): = f(x, y) + c(x^{m N} + y^{m N})$$
is square-free and $\mathbf{P}_N(f) = \mathbf{P}_N(\widetilde f).$
\end{theorem}

Before proving Theorem~\ref{Theorem32}, we need some lemmas of preparation.

\begin{lemma}\label{Lemma33} 
For all $n > \deg f$ and all $c\ne 0$, the polynomial 
$$\widetilde f(x,y):=f(x,y)+c(x^n+y^n)$$
is square-free.
\end{lemma}
\begin{proof} 
Assume for contradiction that $\widetilde f$ is not square free, i.e., there are $p,q\in \mathbb{R}[x,y]$ with $\deg p>0$ such that $\widetilde f(x,y) = p(x, y)^2 q(x, y).$ Let $\widetilde p$ and $\widetilde q$ be the homogeneous components of highest degree of $p$ and $q,$ respectively. We have
$$c(x^n + y^n) = \widetilde p(x, y)^2 \widetilde q(x, y),$$
which is a contradiction since the polynomial $x^n + y^n$ is square-free.
\end{proof}

\begin{lemma}\label{Lemma34}
Let $N > 0$ be an integer. If $y = \varphi(x)$ is a Puiseux series such that $\varphi \mod x^N \not \in \mathbf{P}_N(f),$ then 
$$\ord f(x,\varphi(x)) < mN.$$ 
\end{lemma}
\begin{proof}
By Puiseux's theorem, we can write 
$$f(x,y) = u(x,y) \prod_{k = 1}^m (y - \varphi_k(x)),$$
where $u \in \mathbb{C}\{x, y\}$ with $u(0, 0)\ne 0$ and $\varphi_k,\ k = 1,\dots,m,$ are (not necessarily distinct) Newton--Puiseux roots of $f.$ 
By assumption, $\ord(\varphi - \varphi_k) < N$ for all $k.$ Thus
\begin{eqnarray*}
\ord f(x, \varphi(x)) &=& \sum_{k = 1}^m \ord  (\varphi - \varphi_k) \ < \ m N,
\end{eqnarray*}
which completes the proof.
\end{proof}

\begin{proof} [Proof of Theorem~\ref{Theorem32}]
In view of Lemma~\ref{Lemma33}, the first conclusion is clear.
Let us show that $\mathbf{P}_N(f) \subseteq \mathbf{P}_N(\widetilde f).$ The inverse inclusion is proved analogously.

By contradiction, assume that there exists a Newton--Puiseux root $\varphi$ of $f$ such that $\varphi \mod x^N \not \in \mathbf{P}_N(\widetilde f).$
By definition, $\ord \widetilde{f} = \ord f = m.$ Applying Lemma~\ref{Lemma34} for $\widetilde f,$ we get $\ord \widetilde f(x,\varphi(x)) < mN.$
On the other hand, we have
\begin{eqnarray*}
\widetilde f(x, \varphi(x)) &=& f(x,\varphi(x)) + c \left( x^{mN}+(\varphi(x))^{mN} \right) \ = \ c \left(x^{mN} + (\varphi(x))^{mN} \right).
\end{eqnarray*}
Therefore
\begin{eqnarray*}
mN &>& \ord \widetilde f(x,\varphi(x)) \ = \ \ord \left (x^{mN} + (\varphi(x))^{mN} \right) \ \geqslant \ mN,
\end{eqnarray*}
which is impossible. The theorem is proved.
\end{proof}

In order to work with truncated Newton--Puiseux roots, we need to determine the order from which the roots should be truncated so that the following requirements are fulfilled:
\begin{itemize}
\item The truncations of two distinct roots must be different.
\item It is possible to determine all real roots.
\item The order of a given polynomial $f$ on a Newton--Puiseux root of another polynomial not being a Newton--Puiseux root of $f$ is the same as the order of $f$ on the truncation of this root.
\item It is possible to determine the multiplicity of all roots.
\item It is possible to determine the common roots of two polynomials.
\end{itemize}
We will show, in Theorems~\ref{Theorem36},~\ref{Theorem37}~and~\ref{Theorem39} below, that the truncation order needed is bounded by a number $\mathcal N({\cdot})$ depending only on the degree of the polynomials considered.
Furthermore, according to~\cite{Hoeij1994}, the truncation order needed can be lower than this number since it depends on each polynomial function.

We begin with the following observation, which is inspired by~\cite[Theorem 4.3]{NguyenHD2019}.

\begin{lemma}\label{Lemma35}
Let $f \in \mathbb{R}[x, y]$ be $y$-regular. If $y = \gamma(x)$ is a Newton--Puiseux root of $\frac{\partial f}{\partial y}$ but not of $f,$ then 
$$\ord f(x,\gamma(x))\leqslant (\deg f - 1)^2 + 1.$$
\end{lemma}

Before proving the lemma we recall the notion of intersection multiplicity of two plane curve germs (see, for example, \cite{Greuel2007}). Let $f \in \mathbb{C}\{x, y\}$ be irreducible. Then the {\em intersection multiplicity} of any $g \in \mathbb{C}\{x, y\}$ with $f$ is given by
$$i(f, g) := \mathrm{ord}\, g(x(t),y(t)),$$
where $t \mapsto (x(t), y(t))$ is a parametrization for the curve germ defined by $f.$ Here by a {\em parametrization} of the curve germ $f^{-1}(0),$ we mean an analytic map germ
$$\varphi \colon (\mathbb{C}, 0) \rightarrow (\mathbb{C}^2, 0), \quad t \mapsto (x(t), y(t)),$$
with $f \circ \varphi \equiv 0$ and satisfying the following {\em universal factorization property}:  
for each analytic mapping germ $\psi \colon (\mathbb C, 0) \to (\mathbb C^2, 0)$ with $f \circ\psi  \equiv 0,$ there exists a unique analytic mapping germ $\widetilde\psi\colon (\mathbb C, 0) \to (\mathbb C, 0)$ such that $\psi = \varphi \circ \widetilde\psi.$
In general, let $f \in \mathbb{C}\{x, y\}$ be a convergent power series and let $f = f_1^{\alpha_1} \cdots f_r^{\alpha_r}$  be a factorization of $f$ in the ring $\mathbb{C}\{x, y\}$ with $f_i$ being irreducible and pairwise co-prime.
Then the intersection multiplicity of $g$ with $f$ is defined to be the sum
$$i(f_1^{\alpha_1} \cdots f_r^{\alpha_r}, g) := \alpha_1 i(f_1, g) + \cdots + \alpha_r i(f_r, g).$$

\begin{proof}[Proof of Lemma~\ref{Lemma35}]
Let $g$ be the irreducible factor of $\frac{\partial f}{\partial y}$ in $\mathbb{C}\{x, y\}$ having $\gamma$ as a Newton--Puiseux root. 
Then $t \mapsto (t^n, \gamma(t^n))$ is a parametrization of the curve germ $g^{-1}(0),$ where $n$ denotes the order of $g.$ 
Note that $i(\frac{\partial f}{\partial x}, g)$ is finite because $\gamma$ is a Newton--Puiseux root of $\frac{\partial f}{\partial y}$ but not of $f.$ Hence
\begin{eqnarray*}
\ord f(x, \gamma(x)) &=& \ord \frac{\partial f}{\partial x}(x, \gamma(x)) + 1\\
&\leqslant& n \cdot \ord \frac{\partial f}{\partial x}(x,\gamma(x)) + 1\\
&=& \ord \frac{\partial f}{\partial x}(t^n, \gamma(t^n)) + 1 = i\left(\frac{\partial f}{\partial x}, g\right) + 1.
\end{eqnarray*}
Let $h \in \mathbb C[x, y]$ be the irreducible component of the polynomial $\frac{\partial f}{\partial y}$ which is divisible by $g$ in $\mathbb{C}\{x, y\}.$ Then $\deg h \leqslant \deg f - 1$ and $h$ does not divide $\frac{\partial f}{\partial x}$ (because $i( \frac{\partial f}{\partial x} , g)$ is finite). 
It follows from Bezout's theorem (see \cite[p. 232]{Brieskorn1986}) that
$$i\left(\frac{\partial f}{\partial x},h\right) \leqslant (\deg f - 1)\deg h \leqslant (\deg f - 1)^2.$$
Therefore,
\begin{eqnarray*}
\ord f(x, \gamma(x)) &\leqslant& i \left(\frac{\partial f}{\partial x}, g\right) + 1 \ \leqslant \ i \left(\frac{\partial f}{\partial x}, h \right) + 1 \ \leqslant \ (\deg f - 1)^2 + 1,
\end{eqnarray*}
which completes the proof.
\end{proof}

For $f \in \mathbb{R}[x, y],$ let
\begin{eqnarray*}
\mathbf{P}^{+}(f) &:=& \{\varphi \in \mathbf{P}(f) :\ \varphi \text{ is real}\}, \\
\mathbf{P}^{-}(f) &:=& \mathbf{P}^{+}(f(-x, y)).
\end{eqnarray*}
Moreover for $*\in\{\pm\}$ and $N > 0$ we set
\begin{eqnarray*}
\mathbf{P}^{*}_N(f) &:=& \{\varphi \mod x^N:\ \varphi \in \mathbf{P}^{*}(f)\}.
\end{eqnarray*}
For each integer $d > 0,$ let
\begin{eqnarray*}
\mathcal N(d) &:=& \frac{(d - 1)^2 + 1}{2}.
\end{eqnarray*}

\begin{theorem}\label{Theorem36}
Let $f \in \mathbb{R}[x, y]$ be $y$-regular and $N > \mathcal N(\deg f).$ Then the following statements hold:
\begin{enumerate}[{\rm (i)}]
\item If $y = \varphi_1(x)$ and $y = \varphi_2(x)$ are distinct Newton--Puiseux roots of $f,$ then 
\begin{eqnarray*}
\ord(\varphi_1 - \varphi_2) & \leqslant & \mathcal N(\deg f).
\end{eqnarray*}

\item If $y = \varphi(x)$ is a Newton--Puiseux root of $f$  such that all the coefficients of $\varphi\mod x^N$ are real, then $\varphi$ is real.

\item The mapping 
$$\mathbf{P}(f) \to \mathbf{P}_N(f), \quad \varphi \mapsto \varphi \mod x^N,$$ 
is one-to-one. Furthermore, for $*\in\{\pm\}$, the restriction of this mapping on $\mathbf{P}^{*}(f)$, is a bijection onto $\mathbf{P}_N^{*}(f).$

\item 
Let $\nu(\cdot)$ be the term of lowest order of the considered series.
If $\varphi$ is a Puiseux series such that $\varphi_N := \varphi\mod x^N\not\in \mathbf{P}_N(f),$ then $\nu(f(x, \varphi(x)))=\nu(f(x, \varphi_N(x)).$
In particular,
\begin{eqnarray*}
\ord f(x, \varphi(x)) &=& \ord f(x, \varphi_N(x)).
\end{eqnarray*}

\end{enumerate}
\end{theorem}

\begin{proof} 
(i) By~\cite[Lemma 3.3]{Kuo1977} (see also~\cite[Corollary 3.1]{NguyenHD2019}), there is a Newton--Puiseux root $\gamma$ of $\frac{\partial f}{\partial y}$ but not of $f$ such that
\begin{eqnarray*}
 \ord(\varphi_1 - \varphi_2) &=& \ord(\gamma - \varphi_1) \ = \ \ord(\gamma-\varphi_2).
\end{eqnarray*}
Consequently,
\begin{eqnarray*}
2\, \ord(\varphi_1 - \varphi_2) &=&  \ord(\gamma - \varphi_1) + \ord(\gamma - \varphi_2) \ \leqslant \ \ord f(x,\gamma(x)),
\end{eqnarray*}
which, together with Lemma~\ref{Lemma35}, yields the desired conclusion.

(ii) Assume for contradiction that $\varphi$ is not real.
Since the polynomial $f$ has real coefficients, the complex conjugate $\overline{\varphi}$ of $\varphi$ is also a Newton--Puiseux root of $f$ and 
$\overline{\varphi} \ne \varphi.$ Hence 
\begin{eqnarray*}
\ord (\overline{\varphi} - \varphi) \geqslant N > \mathcal N(d),
\end{eqnarray*} 
which contradicts Item~(i).

(iii) If there are distinct Newton--Puiseux roots $\varphi_1$ and $\varphi_2$ of $f$ such that 
\begin{eqnarray*}
\varphi_1 \mod x^N &=& \varphi_2 \mod x^N,
\end{eqnarray*}
then 
\begin{eqnarray*}
\ord(\varphi_1 - \varphi_2) &\geqslant& N \ > \ \mathcal N(d),
\end{eqnarray*}
which contradicts Item~(i). Hence the mapping
$$\mathbf{P}(f) \to \mathbf{P}_N(f), \quad \varphi \mapsto \varphi \mod x^N,$$ 
is injective, and so is bijective. This, together with Items (i) and (ii), implies that the restriction of this mapping on $\mathbf{P}^{*}(f)$ is a bijection onto $\mathbf{P}_N^{*}(f).$

(iv) 
Since $\varphi_N\not\in \mathbf{P}_N(f)$, we have for all $\psi \in \mathbf{P}(f),$
\begin{eqnarray*}
\nu(\varphi - \psi) &=& \nu(\varphi_N - \psi).
\end{eqnarray*}
On the other hand, by Puiseux's theorem, we have 
$$f(x,y) = u(x,y) \prod_{\psi\in \mathbf{P}(f)} (y - \psi(x))^{\mult_f(\psi)},$$
where $u \in \mathbb{C}\{x, y\}$ with $u(0, 0)\ne 0$.
Therefore
\begin{eqnarray*}
\nu\left(f(x,\varphi(x))\right) &=& \nu\left(u(x,\varphi(x)) \prod_{\psi\in \mathbf{P}(f)} (\varphi(x) - \psi(x))^{\mult_f(\psi)}\right)\\
&=& u(0,0) \prod_{\psi\in \mathbf{P}(f)} \left(\nu(\varphi-\psi)\right)^{\mult_f(\psi)}\\
&=& u(0,0) \prod_{\psi\in \mathbf{P}(f)} \left(\nu(\varphi_N-\psi)\right)^{\mult_f(\psi)}\\
&=& \nu\left(u(x,\varphi_N(x)) \prod_{\psi\in \mathbf{P}(f)} (\varphi_N(x) - \psi(x))^{\mult_f(\psi)}\right)\\
&=&\nu\left(f(x,\varphi_N(x))\right).
\end{eqnarray*}
This ends the proof of the theorem.
\end{proof}

\begin{theorem}\label{Theorem37}
Let $f, g \in \mathbb{R}[x, y]$ be $y$-regular. For all $N > \mathcal N(\deg f + \deg g),$ the mapping 
$$\mathbf{P}(f) \to \mathbf{P}_N(f), \quad \varphi \mapsto \varphi \mod x^N,$$ 
induces a bijection between $\mathbf{P}(f)\setminus \mathbf{P}(g)$ and $\mathbf{P}_N(f)\setminus \mathbf{P}_N(g).$
Furthermore, for $*\in\{\pm\}$, the restriction of this mapping on $\mathbf{P}^{*}(f) \setminus \mathbf{P}^{*}(g)$ is a bijection onto $\mathbf{P}_N^{*}(f)\setminus \mathbf{P}_N^{*}(g).$
\end{theorem}

\begin{proof} 
By definition, if $\varphi \in \mathbf{P}_N(f) \cap \mathbf{P}_N(g)$ then $\varphi \mod x^N \in \mathbf{P}_N(f) \cap \mathbf{P}_N(g).$
Similarly, if $\varphi \in \mathbf{P}^{*}_N(f) \cap \mathbf{P}^{*}_N(g)$ then $\varphi\mod x^N\in \mathbf{P}_N^{*}(f) \cap \mathbf{P}_N^{*}(g).$

Conversely, assume that $\varphi_N \in \mathbf{P}_N(f)\cap \mathbf{P}_N(g).$ Then by the choice of $N$ and by Theorem~\ref{Theorem36}(iii), there is a unique Newton--Puiseux root $\phi$ of $f$ and a unique Newton--Puiseux root $\psi$ of $g$ such that 
\begin{eqnarray*}
\varphi_N &=& \phi \mod x^N \ = \ \psi \mod x^N.
\end{eqnarray*}
Note that $\phi$ and $\psi$ are Newton--Puiseux roots of the product $f \cdot g,$ so again by the choice of $N$ and by Theorem~\ref{Theorem36}(iii), we must have $\phi = \psi.$
In addition, if $\varphi_N\in \mathbf{P}_N^{*}(f) \cap \mathbf{P}_N^{*}(g),$ then $\phi$ and $\psi$ are real in light of Theorem~\ref{Theorem36}(ii).
Therefore the corresponding $\varphi \mapsto \varphi \mod x^N$ induces a bijection between $\mathbf{P}(f) \cap \mathbf{P}(g)$ and $\mathbf{P}_N(f) \cap \mathbf{P}_N(g)$ and a bijection between $\mathbf{P}^{*}(f) \cap \mathbf{P}^{*}(g)$ and $\mathbf{P}_N^{*}(f) \cap \mathbf{P}_N^{*}(g).$

Finally, observe that
\begin{eqnarray*}
\mathbf{P}(f)\setminus \mathbf{P}(g) &=& \mathbf{P}(f) \setminus (\mathbf{P}(f) \cap \mathbf{P}(g)),\\
\mathbf{P}_N(f) \setminus \mathbf{P}_N(g) &=&  \mathbf{P}_N(f) \setminus (\mathbf{P}_N(f)\cap \mathbf{P}_N(g)),\\
\mathbf{P}^{*}(f) \setminus \mathbf{P}^{*}(g) &=&  \mathbf{P}^{*}(f) \setminus (\mathbf{P}^{*}(f)\cap \mathbf{P}^{*}(g)), \\
\mathbf{P}_N^{*}(f) \setminus \mathbf{P}_N^{*}(g) &=& \mathbf{P}_N^{*}(f)\setminus (\mathbf{P}_N^{*}(f)\cap \mathbf{P}_N^{*}(g)).
\end{eqnarray*}
Hence, the desired conclusion follows from the choice of $N$ and Theorem~\ref{Theorem36}(iii).
\end{proof}

\begin{lemma}\label{Lemma38}
Let $f \in \mathbb{R}[x, y]$ be $y$-regular and $\varphi$ be a Newton--Puiseux root of $f.$ For all $N > \mathcal N(\deg f),$ the following conditions are equivalent:
\begin{enumerate}[{\rm (i)}]
\item $\varphi$ is of multiplicity greater than $1$ (as a Newton--Puiseux root of $f$).
\item $\varphi$ is a Newton--Puiseux root of $\frac{\partial f}{\partial y}.$ 
\item $\varphi \mod x^N \in \mathbf{P}_N(\frac{\partial f}{\partial y}).$
\end{enumerate}
\end{lemma}

\begin{proof} 
The implications (i) $\Rightarrow $ (ii) $\Rightarrow $ (iii) are clear.

(iii) $\Rightarrow $ (i): By assumption, there is a Newton--Puiseux root $\gamma$ of $\frac{\partial f}{\partial y}$ such that 
\begin{eqnarray*}
\ord (\gamma - \varphi) &\geqslant& N.
\end{eqnarray*}
In view of \cite[Lemma 3.3]{Kuo1977} (see also~\cite[Corollary 3.1]{NguyenHD2019}), there is a Newton--Puiseux root $\psi$ of $f$ such that 
\begin{eqnarray*}
\ord (\psi - \varphi) &=& \ord (\gamma - \varphi) \ \geqslant \ N \ > \ \mathcal{N}(\deg f),
\end{eqnarray*}
which, together with Theorem~\ref{Theorem36}(i), yields $\psi = \varphi,$ and so $\gamma = \varphi.$
\end{proof}

\begin{theorem}\label{Theorem39}
Let $f \in \mathbb{R}[x, y]$ be $y$-regular, $\varphi$ be a Newton--Puiseux root of $f$ and $p \geqslant 1.$
For all $N > \mathcal N(\deg f),$ the following conditions are equivalent:
\begin{enumerate}[{\rm (i)}]
\item $\varphi$ is of multiplicity $p$ (as a Newton--Puiseux root of $f$).

\item $\varphi$ is a Newton--Puiseux root of $\frac{\partial^{k} f}{\partial y^{k}}$ for all non negative integers $k \leqslant p - 1$ but not of $\frac{\partial^{p} f}{\partial y^{p}}.$

\item $\varphi\mod x^N\in \mathbf{P}_N(\frac{\partial^{k} f}{\partial y^{k}})$ for all non negative integers $k \leqslant p - 1$ and $\varphi \mod x^N \not \in \mathbf{P}_N(\frac{\partial^{p} f}{\partial y^{p}}).$
\end{enumerate}
\end{theorem}

\begin{proof} 
(i) $\Leftrightarrow $ (ii): This is clear.

(ii) $\Rightarrow $ (iii): Obviously, $\varphi \mod x^N \in \mathbf{P}_N(\frac{\partial^{k} f}{\partial y^{k}})$ for $k\leqslant p - 1.$
Furthermore, $\varphi \mod x^N \not \in \mathbf{P}_N(\frac{\partial^{p} f}{\partial y^{p}})$ because otherwise it follows from 
Lemma~\ref{Lemma38} that $\varphi$ is a Newton--Puiseux root of $\frac{\partial^{p} f}{\partial y^{p}},$ which is a contradiction.
Therefore (iii) holds.

(iii) $\Rightarrow $ (ii):
Let us show that $\varphi$ is a Newton--Puiseux root of $\frac{\partial^{k} f}{\partial y^{k}}$ for all non negative integer $k\leqslant p-1.$
We proceed by induction on $k.$
Clearly the statement holds for $k = 0.$
Assume that the statement holds for $k = l \geqslant 0.$
We need to show that it also holds for $k = l + 1 \leqslant p - 1.$
But, in view of Lemma~\ref{Lemma38}, this is a direct consequence of the induction assumption, the assumption $\varphi \mod x^N \in \mathbf{P}_N(\frac{\partial^{l+1} f}{\partial y^{l+1}})$ and the inequalities 
$$N > \mathcal N(\deg f) \geqslant \mathcal N\left(\deg \frac{\partial^{l} f}{\partial y^{l}}\right).$$
Finally, since $\varphi\mod x^N\not\in \mathbf{P}_N(\frac{\partial^{p} f}{\partial y^{p}}),$ it follows from Lemma~\ref{Lemma38} that $\varphi$ is not a Newton--Puiseux root of $\frac{\partial^{p} f}{\partial y^{p}}.$
\end{proof}

The next corollary shows that multiplicities of Newton--Puiseux roots can be computed based on their truncations.

\begin{corollary}
Let $f \in \mathbb{R}[x, y]$ be $y$-regular, $\varphi$ be a Newton--Puiseux root of $f$ and $\varphi_N := \varphi \textrm{ mod } x^N.$ Then for any $N > \mathcal{N}(\deg {f})$ we have
\begin{eqnarray*}
\mult_f(\varphi) &=& \mathrm{ord} f(x, \varphi_{N + 1} + cx^{N + 1}) - \mathrm{ord} f(x, \varphi_{N} + cx^{N}),
\end{eqnarray*}
where $c$ is a generic constant.
\end{corollary}

\begin{proof}
By definition, for generic $c,$ the following two equalities are satisfied:
\begin{eqnarray*}
\ord (\varphi_N + c x^N - \varphi) & = & N, \\
\ord (\varphi_{N + 1} + c x^{N + 1} - \varphi) & = & N + 1.
\end{eqnarray*}
In light of Theorem~\ref{Theorem36}(i), we have for any $\psi \in \mathbf{P}(f) \setminus \{\varphi\},$
\begin{eqnarray*}
 \ord (\varphi - \psi) &=& \ord (\varphi_N + c x^N - \psi) \ = \  \ord (\varphi_{N + 1} + c x^{N + 1} - \psi).
\end{eqnarray*}
On the other hand, it follows from Puiseux's theorem that
$$f(x, y) = f_1(x, y) (y - \varphi(x))^{\mult_f(\varphi)},$$
where 
\begin{eqnarray*}
f_1(x, y) &=& u(x, y) \prod_{\psi \in \mathbf{P}(f) \setminus \{\varphi\}} \big(y - \psi(x) \big)^{\mult_f(\psi)}
\end{eqnarray*}
with $u \in \mathbb{C}\{x, y\}$ and $u(0, 0) \ne 0.$ Hence
\begin{eqnarray*}
\ord f_1(x, \varphi_N + c x^N) 
&=& \sum_{\psi \in \mathbf{P}(f) \setminus \{\varphi\}} \ord (\varphi_N + c x^N - \psi) \, {\mult_f(\psi)} \\
&=& \sum_{\psi \in \mathbf{P}(f) \setminus \{\varphi\}} \ord (\varphi_{N + 1} + c x^{N + 1} - \psi) \, {\mult_f(\psi)} \\
&=& \ord f_1(x, \varphi_{N + 1} + c x^{N + 1}). 
\end{eqnarray*}
Therefore,
\begin{eqnarray*}
\mult_f(\varphi) &=& \ord f(x, \varphi_{N + 1} + c x^{N + 1}) - \ord f(x, \varphi_N + c x^N),
\end{eqnarray*}
which is the desired conclusion.
\end{proof}

\begin{remark}{\rm
Let $f \in \mathbb{R}[x, y]$ be a non-constant polynomial such that the origin is a non isolated zero of $f.$
Then for small $r > 0,$ the intersection of the real plane algebraic curve $f^{-1}(0)$ with the punctured disk $\mathbb{B}_r \setminus \{(0, 0)\}$ has a fixed number of connected components, each one homeomorphic to a line. The germ at the origin of such a connected component will be called {\em half-branch} at the origin of $f^{-1}(0).$ 

Assume that $f$ is $y$-regular and let $C$ be a half-branch at the origin of the curve $f^{-1}(0)$ in $\mathbb{R}^2.$ Then $C$ is not tangent to the $y$-axis. Hence, the first coordinate $x$ tends to $0^+$ or $0^{-}$ along $C.$ We say that $C$ is a {\em right half-branch} in the first case and a {\em left half-branch} in the second case. If $C$ is a right (resp., left) half-branch, there exist $\delta > 0$ and a semi-algebraic continuous function $\varphi \colon [0, \delta) \to \mathbb{R}$ with $\varphi(0) = 0$ such that $C$ is the germ of the curve $(x = t, y = \varphi(t))$ (resp., $(x = -t, y = \varphi(t))$) as $t \to 0^{+};$ moreover we have $\varphi(t)=\sum_{k\geqslant 0}c_k t^{\frac{n_k}{n}} \in \mathbf{P}^{+}(f)$ (resp., $\varphi \in \mathbf{P}^{-}(f)$). 
 
Conversely, each Puiseux series $\varphi \in \mathbf{P}^{+}(f)$ (resp., $\varphi \in \mathbf{P}^{-}(f)$) determines a right (resp., left) half-branch at the origin of the curve $f^{-1}(0).$
}\end{remark}

\section{The existence of limits of real bivariate rational functions} \label{Section4}

In this section, we present necessary and sufficient conditions for the existence of limits of real bivariate rational functions. 

Let $f, g \in \mathbb{R}[x, y]$ be nonzero polynomials and $(a, b) \in \mathbb{R}^2.$ 
The existence of the limit $\lim_{(x, y) \to (a, b)}\frac{f(x, y)}{g(x, y)}$ does not depend on the particular choice of local coordinates. Hence, after a suitable translation we may assume that the given point $(a, b)$ is the origin $(0, 0) \in \mathbb{R}^2.$ If $g(0, 0)  \ne 0$ then the limit exists and equals to $\frac{f(0, 0)}{g(0, 0)}.$ 
If $f(0, 0)  \ne 0$ and $g(0, 0) = 0,$ then the limit does not exist. Therefore without loss of generality we will assume that $f(0, 0) = g(0, 0) = 0.$ 
For simplicity, let $m := \ord f, n := \ord g,$ and let $f_m, g_n$ be the homogeneous components of degree $m$ and $n$ of $f$ and $g,$ respectively.
Note that $m$ and $n$ are positive integers.

Let us start with the following simple observation.
\begin{lemma} \label{Lemma41}
Assume that the limit $\lim_{(x, y) \to (0, 0)} \frac{f(x, y)}{g(x, y)}$ exists and equals to $L \in \mathbb{R}.$ Then one of the following conditions holds
\begin{enumerate}[{\rm (i)}]
\item $m > n$ and $L = 0.$ 
\item $m  = n,$ $L \ne 0$ and $f_m \equiv L g_n.$
\end{enumerate}
Moreover, if $f$ and $g$ have no non-constant common divisor in $\mathbb{R}[x, y]$, then $g$ must have an isolated zero at the origin.
\end{lemma}

\begin{proof}
Let $L := \lim_{(x, y) \to (0, 0)} \frac{f(x, y)}{g(x, y)} \in \mathbb{R}.$ 
Take any $(\bar{x}, \bar{y}) \in \mathbb{R}^2 \setminus (f_m^{-1}(0) \cup g_n^{-1}(0)).$ Then for all $t > 0$ small enough, we have
$g(t \bar{x}, t \bar{y}) \ne 0$ and so
\begin{eqnarray*}
L \ = \ \lim_{t \to 0^+} \frac{f(t \bar{x}, t \bar{y})}{g(t \bar{x}, t \bar {y})}  
 &=&\lim_{t \to 0^+} \frac{t^m f_m(\bar{x}, \bar{y}) + \textrm{ higher order terms in } t}{t^n g_n(\bar{x}, \bar{y}) + \textrm{ higher order terms in } t} \\
 &=&
\begin{cases}
0 & \textrm{ if } m > n, \\
\frac{f_m(\bar{x}, \bar{y})}{g_n(\bar{x}, \bar{y})} & \textrm{ if } m = n, \\
\infty & \textrm{ otherwise.}
\end{cases}
\end{eqnarray*}
Hence $m \geqslant n$ and the strict inequality occurs only if $L = 0.$ 

Assume $m = n.$ Then $L = \frac{f_m(\bar{x}, \bar{y})}{g_n(\bar{x}, \bar{y})} \ne 0.$ Since this holds for any $(\bar{x}, \bar{y}) \in \mathbb{R}^2 \setminus (f_m^{-1}(0) \cup g_n^{-1}(0)),$ it is easy to see that $f_m(x, y)  = L g_n (x, y)$ for all $(x, y) \in \mathbb{R}^2.$

Now, suppose for contradiction that $g$ has a non isolated zero at the origin.
Since $f$ and $g$ have no non-constant common divisor in $\mathbb{R}[x, y]$, in view of \cite[Proposition 4, page~179]{Brieskorn1986} (see also \cite[Algebraic Corollary. 3]{Puente2002}), there exist $h \in \mathbb{R}[x] \backslash\{0\}$ and $p, q \in \mathbb{R}[x,y]$ such that $h=p f+q g$.
Fix $\epsilon > 0.$ There exists $\delta > 0$ such that for all $(x, y) \in \mathbb{B}_{\delta},$ 
$$|f(x, y) - L g(x, y) | \leqslant \epsilon |g(x, y)|.$$
In particular, $g^{-1}(0) \cap \mathbb{B}_{\delta} \subset f^{-1}(0).$ 
Consider $h$ as a polynomial in $x$ and $y$ but does not depend on $y$. 
Then 
$$g^{-1}(0)\cap \mathbb{B}_{\delta}\subset h^{-1}(0).$$
Since $g$ has a non isolated zero at the origin and $\{x=0\}$ is the only component of $h^{-1}(0)$ passing through the origin, it follows that $g^{-1}(0)\cap \mathbb{B}_{\delta}$ contains infinitely many points of $\{x=0\}$.
Thus $\{x=0\}\subset g^{-1}(0)$ and so $\{x=0\}\subset f^{-1}(0).$
Therefore $f$ and $g$ must have the common divisor $x$ which contradicts the assumption.
Consequently $g$ has an isolated zero at the origin.
\end{proof}

\begin{remark}[see also \cite{Zeng2020}]{\rm
If $f$ (or $g$) is $y$-regular, then $f$ and $g$ have no non-constant common divisor in $\mathbb{R}[x, y]$ if and only if 
$f$ and $g$ have no non-constant common divisor in $\mathbb{R}[x, y] \setminus \mathbb{R}[x].$ 
}\end{remark}

In light of Lemma~\ref{Lemma41}, replacing $f$ by $f - \frac{f_m(\bar{x}, \bar{y})}{g_n(\bar{x}, \bar{y})} g$ with  $(\bar{x}, \bar{y}) \in \mathbb{R}^2 \setminus (f_m^{-1}(0) \cup g_n^{-1}(0)),$ if $m = n,$ our problem can be reduced to that of determining if the limit $\lim_{(x, y) \to (0, 0)} \frac{f(x, y)}{g(x, y)}$ exists and equals to $0.$

\subsection{Case of isolated zeros of the denominator}

In this subsection, we will consider the case where the polynomial $g$ has {\em an isolated zero at the origin.}

\begin{lemma}\label{Lemma43}
Let $g$ have an isolated zero at the origin and assume that $g \geqslant 0$ in some neighborhood of the origin. There exist positive constants $r$ and $\delta$ such that for all $t \in (0, \delta),$ the following statements hold:
\begin{enumerate} [{\rm (i)}]
\item $t$ is a regular value of $g.$ 

\item $g^{-1}(t)\cap \mathbb B_r$ is a non-singular connected closed curve bounding a region containing the origin.
\end{enumerate}
\end{lemma}

\begin{proof} 
By Sard's theorem, the set of critical values of $g$ is finite. In particular, there is $\delta_0 > 0$ such that $(0, \delta_0)$ does not contain a critical value of $g.$ Furthermore, by assumption, there exists $r > 0$ such that $g^{-1}(0)\cap{\mathbb B}_r = \{(0, 0)\}$ and
\begin{eqnarray*}
0 & < & \sup_{z\in \mathbb B_{r}} g(z) \ < \ \delta_0.
\end{eqnarray*}
Let $\delta > 0$ be such that $\displaystyle \inf_{z\in\mathbb S_r} g(z)  > \delta.$ By definition, $[0, \delta) \subset g(\mathbb B_{r}).$ Take any $t \in (0, \delta).$ Clearly, $t$ is a regular value of $g$ and satisfies
\begin{eqnarray*}
g^{-1}(t) \cap{\mathbb B}_r &=& g^{-1}(t) \cap \mathrm{int}( \mathbb{B}_r)  \ \ne \ \emptyset.
\end{eqnarray*}
Consequently, the set $g^{-1}(t) \cap{\mathbb B}_r$ is a compact one-dimensional manifold and so it is a disjoint union of  non-singular closed curves.

If there is a closed curve $\alpha$ in the union bounding an open and bounded region $M$ not containing the origin, then the polynomial $g$ must have a local extremum $z_1$ in $M$ and hence $z_1$ is a critical point of $g.$ We have
$$0 < g(z_1) \leqslant \sup_{z \in M} g(z) \leqslant \sup_{z \in \mathbb B_{r}} g(z) < \delta_0.$$ 
This contradicts the fact that $(0, \delta_0)$ does not contain a critical value of $g.$

Thus $g^{-1}(t)\cap \mathbb B_r$ is a disjoint union of closed curves, each bounds an open and bounded region containing $0.$
Let $\alpha_1,\alpha_2 \subset g^{-1}(t)\cap \mathbb B_r$ be such two curves and let $M_1,M_2$ be the open and bounded regions bounded by them, respectively. We must have either $M_1\subsetneq M_2$ or $M_2\subsetneq M_1.$
Without loss of generality, assume that $M_1\subsetneq M_2$.
It is clear that the polynomial $g$ must have a local extremum $z_2$ in $M_2\setminus (M_1\cup\alpha_2)$ and hence $z_2$ is a critical point of $g.$
Observe that
$$0 < g(z_2) \leqslant \sup_{z \in M_2} g(z) \leqslant \sup_{z \in \mathbb B_{r}} g(z) < \delta_0.$$ 
This also contradicts the fact that $(0, \delta_0)$ does not contain a critical value of $g.$ Therefore $g^{-1}(t)\cap\mathbb B_r$ must be a (non-singular) connected closed curve bounding a region containing $0.$ The lemma is proved.
\end{proof}

As in the introduction, let $F_{f, g}$ be the determinant of the Jacobian matrix of the mapping 
$$\mathbb{R}^2 \to \mathbb{R}^2, \quad (x, y) \mapsto (f(x, y), g(x, y)),$$
 i.e., 
$$F_{f, g} :=\frac{\partial f}{\partial x}\frac{\partial g}{\partial y}-\frac{\partial f}{\partial y}\frac{\partial g}{\partial x}.$$

\begin{remark}{\rm
(i) By definition, $F_{f,  g}(x, y) = 0$ if and only if the gradient vectors $\nabla f(x, y) $ and $\nabla g(x, y) $ are linearly dependent. 

(ii) If the origin is an isolated zero of $g,$ then it is a local extremum of $g;$ in particular, $\nabla g(0, 0) = 0$ and so $F_{f, g}(0, 0) = 0.$
}\end{remark}

\begin{lemma}\label{Lemma45}
Assume that $F_{f, g} \equiv 0.$ If $t \in \mathbb{R}$ is a regular value of $g$ then the restriction of $f$ on each connected component of $g^{-1}(t)$ is constant.
\end{lemma}
\begin{proof}
Let $S$ be a connected component of $g^{-1}(t)$ and fix $z_1, z_2 \in S.$ We need to show $f(z_1) = f(z_2).$
Indeed, since $S$ is connected and semi-algebraic, there exists a piecewise smooth and continuous semi-algebraic curve $\alpha \colon [0, 1] \to \mathbb R^2$ lying in $S$ such that $\alpha(0) = z_1$ and $\alpha(1) = z_2.$ Breaking the curve into smooth pieces, it suffices to consider the case where the curve is smooth. Observe that for all $s \in [0, 1],$ we have $g(\alpha(s)) = t,$ $\nabla g(\alpha(s)) \ne 0$ and $F_{f, g}(\alpha(s)) = 0.$ Hence, there exists $\lambda(s) \in \mathbb{R}$ such that
\begin{eqnarray*}
\nabla f(\alpha({s}))  &=& \lambda (s) \nabla g(\alpha({s})).
\end{eqnarray*}
It follows that
\begin{eqnarray*}
\frac{d }{d{s}}(f \circ \alpha)({s})
&=& \left\langle \nabla  f(\alpha({s})), \frac{d \alpha({s})}{d{s}} \right\rangle \\
&=& \lambda(s) \left\langle \nabla  g(\alpha({s})), \frac{d \alpha({s})}{d{s}} \right\rangle \\
&=& \lambda({s}) \frac{d}{d{s}}(g \circ \alpha)({s}) \ = \ 0.
\end{eqnarray*}
So ${f}$ is constant on the curve $\alpha;$ in particular, $f(z_1) = f(z_2).$
\end{proof}

\begin{theorem}
Assume that $g$ has an isolated zero at the origin and that $F_{f, g} \equiv 0.$ We have
$$\displaystyle\lim_{(x, y) \to (0, 0)}\frac{f(x, y)}{g(x, y)} = 0 \quad \textrm{ if and only if } \quad \mathrm{ord} f > \mathrm{ord} g.$$
\end{theorem}

\begin{proof}
By Lemma~\ref{Lemma41}, it suffices to prove the sufficient condition. To this end, assume that $\mathrm{ord} f > \mathrm{ord} g$ and let $(x_k, y_k) \in \mathbb{R}^2 \setminus g^{-1}(0)$ be any sequence tending to $(0, 0).$ We need to show 
$$\displaystyle\lim_{k \to \infty} \frac{f(x_k, y_k)}{g(x_k, y_k)} = 0.$$

There is no loss of generality in assuming that $f$ and $g$ are $y$-regular. 
Indeed, by Lemma~\ref{Lemma31}, there is $c\in\mathbb R$ such that $\widetilde f(x,y):=f(x+cy,y)$ and $\widetilde g(x,y):=g(x+cy,y)$ are $y$-regular. 
We note the following facts:
\begin{itemize}
\item $\mathrm{ord} \widetilde f=\mathrm{ord} f;$
\item $\mathrm{ord} \widetilde g=\mathrm{ord} g;$
\item $\displaystyle\lim_{(x, y) \to (0, 0)}\frac{f(x, y)}{g(x, y)} = 0$ if and only if $\displaystyle\lim_{(x, y) \to (0, 0)}\frac{\widetilde f(x, y)}{\widetilde g(x, y)} = 0;$ and
\item $
F_{\widetilde f,\widetilde g}(x,y)\ =\ \left[\frac{\partial f}{\partial x}\left(c\frac{\partial g}{\partial x}+\frac{\partial g}{\partial y}\right)-\left(c\frac{\partial f}{\partial x}+\frac{\partial f}{\partial y}\right)\frac{\partial g}{\partial x}\right](x+cy,y)\ =\ F_{f,g}(x+cy,y)\ = \ 0
$
for all $(x,y)\in\mathbb R^2.$
\end{itemize}
Hence, if $f$ or $g$ is not $y$-regular, we replace $f$ and $g$ by $\widetilde f$ and $\widetilde g$ respectively.
Furthermore, replacing $g$ by $-g$ if necessary, we may assume that $g \geqslant 0$ in some neighborhood of the origin $(0, 0) \in \mathbb{R}^2.$
Let $r > 0$ and $\delta> 0$ be the constants determined by Lemma~\ref{Lemma43}. Then for all sufficiently large $k,$ we have $t_k := g(x_k, y_k) \in (0, \delta),$ and so $t_k$ is a regular value of $g$ and the set $g^{-1}(t_k)\cap \mathbb B_r$ is a non-singular connected closed curve bounding a region containing $(0, 0).$ The latter fact implies that $\big(g^{-1}(t_k) \cap \mathbb B_r \big) \cap \{x = 0\}$ contains at least a point $(0, \bar y_k).$
Obviously 
$$g(0, \bar y_k) = t_k \to 0 \quad \text{ as } \quad k\to\infty.$$
As the origin $(0, 0)$ is an isolated zero of $g$, we must have $\bar y_k \to 0.$
On the other hand, by Lemma~\ref{Lemma45}, $f(0, \bar y_k) = f(x_k, y_k).$ Therefore,
\begin{eqnarray*}
\lim_{k \to \infty} \frac{f(x_k, y_k)}{g(x_k, y_k)}   &=&  \lim_{k \to \infty} \frac{f(0, \bar y_k)}{g(0, \bar y_k)} \ = \ 0,
\end{eqnarray*}
where the second equality follows from the facts that $f$ and $g$ are $y$-regular of order $\mathrm{ord} f$ and $\mathrm{ord} g$ respectively with
$\mathrm{ord} f > \mathrm{ord} g.$
\end{proof}

We next consider the case $F_{f, g} \not \equiv 0.$ To this end, we need the following fact.

\begin{lemma}
For all but a finite number of $c \in \mathbb{R},$ the polynomials $\widetilde{f}(x,y) := f(x + cy, y),$ $\widetilde{g}(x, y) := g(x + cy, y),$ and  $F_{\widetilde{f}, \widetilde{g}}$ are $y$-regular.
\end{lemma}

\begin{proof}
By definition, we have
$$F_{\widetilde{f}, \widetilde{g}}(x, y) 
\ =\ \left[\frac{\partial f}{\partial x}\left(c \frac{\partial g}{\partial x} + \frac{\partial g}{\partial y}\right) - \left(c \frac{\partial f}{\partial x} + \frac{\partial f}{\partial y}\right)\frac{\partial g}{\partial x}\right](x + cy, y)
\ =\ F_{f,g}(x + cy, y).
$$
This, together with Lemma~\ref{Lemma31}, implies the required conclusion. 
\end{proof}

\begin{theorem}\label{Theorem48}
Assume that $g$ has an isolated zero at the origin, $F_{f, g} \not \equiv 0$ and the polynomials $f, g$ and $F_{f, g}$ are $y$-regular. 
Then the following conditions are equivalent:
\begin{enumerate} [{\rm (i)}]
\item $\displaystyle\lim_{(x, y) \to (0, 0)}\frac{f(x, y)}{g(x, y)} = 0.$ 

\item For $*\in\{\pm\}$ and for any $\gamma\in \mathbf{P}^{*}\left(F_{f,g}\right)\setminus \mathbf{P}^{*}(f),$ we have 
\begin{eqnarray*}
\ord{f(* x, \gamma(x))} &>& \ord{g(* x, \gamma(x))}.
\end{eqnarray*}
\end{enumerate}
\end{theorem}

\begin{proof}
(i) $\Rightarrow$ (ii): This is clear.

(ii) $\Rightarrow$ (i): 
Replacing $g$ by $-g$ if necessary, we may assume that $g \geqslant 0$ in some neighborhood of the origin $(0, 0) \in \mathbb{R}^2.$
By Lemma~\ref{Lemma43}, there exist $r > 0$ and $\delta> 0$ such that each $t$ in $(0, \delta)$ is a regular value of $g$ and satisfies
\begin{eqnarray*}
g^{-1}(t) \cap \mathbb{B}_r &=& g^{-1}(t) \cap \mathrm{int} (\mathbb{B}_r) \ \ne \ \emptyset.
\end{eqnarray*}
Thus, for such $t,$ the following equalities hold:
\begin{eqnarray*}
\min_{(x, y) \in g^{-1}(t) \cap \mathbb{B}_r} f(x, y) & = & \min_{(x, y) \in g^{-1}(t) \cap \mathrm{int}(  \mathbb{B}_r)} f(x, y), \\
\max_{(x, y) \in g^{-1}(t) \cap \mathbb{B}_r} f(x, y) & = & \max_{(x, y) \in g^{-1}(t) \cap \mathrm{int}(  \mathbb{B}_r)} f(x, y).
\end{eqnarray*}
Applying Lemma~\ref{LemmaSAChoice} for the (nonempty) semi-algebraic set
\begin{eqnarray*}
S & := &  \Big  \{(t, (x, y)) \in (0, \delta)  \times \mathbb{B}_r \ \Big | \  (x, y) \in \mathrm{argmin}_{(x, y) \in g^{-1}(t) \cap \mathbb{B}_r} f(x, y) \Big \}
\end{eqnarray*}
and the projection $S \to (0, \delta), ({t}, (x, y)) \mapsto {t},$ we get a semi-algebraic curve 
$$\gamma \colon (0, \delta) \to \mathbb{R}^2, \quad t \mapsto (x(t), y(t)),$$ 
such that for all $t \in (0, \delta)$ we have that $\gamma(t) \in g^{-1}(t) \cap \mathrm{int}  (\mathbb{B}_r)$ is an optimal solution of the problem $\min_{(x, y) \in g^{-1}(t) \cap \mathbb{B}_r} f(x, y).$ 
By the Lagrange multiplier theorem, the gradient vectors $\nabla f(\gamma(t))$ and $\nabla g(\gamma(t))$ are linearly dependent, i.e., $F_{f, g}(\gamma(t)) \equiv 0.$ Since $F_{f, g}$ is $y$-regular, $x(t) \ne 0$ for small $t > 0.$ 
By Lemma~\ref{MonotonicityLemma} and by shrinking $\delta$ (if necessary), we may assume that for all $t \in (0, \delta),$ either $x(t) > 0$ or $x(t) < 0.$ Then the germ of the curve $t \mapsto \gamma(t)$ is either a right half-branch or a left half-branch of $F_{f, g}^{-1}(0).$ This, together with the condition~(ii), yields 
\begin{eqnarray*}
\min_{(x, y) \in g^{-1}(t) \cap \mathbb{B}_r} \frac{f(x, y)}{g(x, y)}  &=& \frac{f(\gamma(t))}{g(\gamma(t))} \ \to \ 0
\quad \textrm{ as } \quad t \to 0^+.
\end{eqnarray*}
Similarly, we also have 
\begin{eqnarray*}
\max_{(x, y) \in g^{-1}(t) \cap \mathbb{B}_r} \frac{f(x, y)}{g(x, y)}  &\to & 0 \quad \textrm{ as } \quad t \to 0^+.
\end{eqnarray*}

Let $(x_k, y_k) \in \mathbb{R}^2 \setminus g^{-1}(0)$ be any sequence tending to $(0, 0).$ 
Then for all sufficiently large $k,$ we have $(x_k, y_k) \in \mathbb{B}_r$ and $t_k := g(x_k, y_k) \in (0, \delta).$ For such $k$, the following inequalities hold:
\begin{eqnarray*}
\min_{(x, y) \in g^{-1}(t_k) \cap \mathbb{B}_r} \frac{f(x, y)}{g(x, y)}  &\leqslant& 
\frac{f (x_k, y_k) }{g (x_k, y_k) } \ \leqslant \ \max_{(x, y) \in g^{-1}(t_k) \cap \mathbb{B}_r} \frac{f(x, y)}{g(x, y)}.
\end{eqnarray*}
Therefore, $\displaystyle\lim_{k \to \infty} \frac{f(x_k, y_k)}{g(x_k, y_k)} = 0,$ which completes the proof.
\end{proof}

A truncated version of Theorem~\ref{Theorem48}, which will be used in the design of our algorithms, reads as follows.

\begin{corollary}\label{Corollary49}
Under the assumptions of Theorem~\ref{Theorem48}, let $N$ be an integer such that
\begin{eqnarray*}
N &>& \max\{\mathcal N(\deg  f+\deg F_{f,g}), \mathcal N(\deg  g)\}.
\end{eqnarray*}
Then the following conditions are equivalent:
\begin{enumerate} [{\rm (i)}]
\item $\displaystyle\lim_{(x, y) \to (0, 0)}\frac{f(x, y)}{g(x, y)} = 0.$ 

\item For $*\in\{\pm\}$ and for any $\gamma_N\in \mathbf{P}_N^{*}\left(F_{f,g}\right)\setminus \mathbf{P}_N^{*}(f),$ we have 
\begin{eqnarray*}
\ord f(* x, \gamma_N(x)) &>& \ord g(* x, \gamma_N(x)).
\end{eqnarray*}
\end{enumerate}
\end{corollary}

\begin{proof}
By Theorem~\ref{Theorem37} and the choice of $N,$ the mapping
$$\mathbf{P}^{*}(F_{f,g}) \setminus \mathbf{P}^{*}(f) \to \mathbf{P}_N^{*}(F_{f,g})\setminus \mathbf{P}_N^{*}(f), \quad \gamma \mapsto \gamma \mod x^N,$$ 
is bijective. Therefore, in view of Theorem~\ref{Theorem48}, it suffices to show that for any $\gamma \in \mathbf{P}^{*}\left(F_{f, g}\right)\setminus \mathbf{P}^{*}(f),$ the following equalities hold:
\begin{eqnarray*}
\ord{f(* x, \gamma(x))} &=& \ord{f(* x, \gamma(x)\mod x^N)}, \\
\ord{g(* x, \gamma(x))} &=& \ord{g(* x, \gamma(x)\mod x^N)}.
\end{eqnarray*}

In fact, the first equality follows directly from Theorem~\ref{Theorem36}(iv).
The second inequality follows similarly by remarking that $\mathbf{P}^{*}(g) = \emptyset,$ and so $\gamma \not \in \mathbf{P}^{* }(g)$ for any $\gamma \in \mathbf{P}^{*}\left(F_{f, g}\right).$
\end{proof}

\begin{remark}{\rm
\begin{enumerate}[{\rm (i)}]
\item As we have seen, in the case where $g$ has an isolated zero at the origin, the limit $\displaystyle\lim_{(x, y) \to (0, 0)}\frac{f(x, y)}{g(x, y)}$ exists if and only if it exists along the half-branches of the curve $F_{f, g} = 0.$ 
Similar approach has appeared in \cite{Cadavid2013} for this case. 
However, instead of $F_{f, g},$ whose degree is bounded by $\deg f+\deg g-2$, the polynomial $G_{f,g}$ given by~\eqref{Gfg}, whose degree is bounded by $\deg f+\deg g,$ is used there. So in general, it has greater degree than that of $F_{f, g}.$
\item 
Let $N$ be an integer such that $N > \mathcal N(\deg  g)$.
If $\mathbf{P}^{+}_N({g})\cup\mathbf{P}^{-}_N({g})=\emptyset,$ then it not hard to see that $g$ has a local extremum at the origin $(0,0)$. 
Furthermore, let $c$ be the coefficient of the term of lowest degree in $g(0,y)$. 
Clearly $(0,0)$ is a local maximum of $g$ if $c>0$; otherwise it is a local minimum of $g$.
Assume that $f(0,0)\ne 0.$ 
Then 
$$\displaystyle\lim_{(x, y) \to (0, 0)}\frac{f(x, y)}{g(x, y)}=\left\{
\begin{array}{lll}
+\infty & \text{if} & c f(0,0)>0\\
-\infty & \text{if} & c f(0,0)<0.
\end{array}\right.$$
\end{enumerate}

}\end{remark}

\subsection{The general case}

We now consider the general case: the polynomial $g$ has {\em not necessarily} an isolated zero at the origin. To this end, 
as in the introduction, define the polynomial $G_{f, g} \in \mathbb{R}[x, y]$ by
\begin{eqnarray}\label{Gfg}
G_{f, g} 
&:=& y\left(g\frac{\partial f}{\partial x}- f\frac{\partial g}{\partial x}\right) - x\left(g\frac{\partial f}{\partial y} - f\frac{\partial g}{\partial y}\right).
\end{eqnarray}
Observe that if $(x, y) \not \in g^{-1}(0),$ then $G_{f,  g}(x, y) = 0$ if and only if the vectors $\nabla \big( \frac{f}{g}\big)(x, y)$  and $(x, y)$ are linearly dependent. Also note that $G_{f,  g}$ vanishes at the origin.

Let us start with the case $G_{f, g}  \equiv 0.$

\begin{theorem}
Assume that $G_{f, g} \equiv 0.$ Then
$$\displaystyle\lim_{(x, y) \to (0, 0)}\frac{f(x, y)}{g(x, y)} = 0 \quad \textrm{ if and only if } \quad \mathrm{ord} f > \mathrm{ord} g.$$
\end{theorem}

\begin{proof}
By Lemma~\ref{Lemma41}, it suffices to prove the sufficient condition.
So, let $\mathrm{ord} f > \mathrm{ord} g.$ 

Fix $t > 0.$ Since $G_{f, g} \equiv 0,$ it is not hard to check that the restriction of $\frac{f}{g}$ on each connected component of $\mathbb{S}_t \setminus g^{-1}(0)$ is constant; confer Lemma~\ref{Lemma45}. Take any $(\bar{x}, \bar{y}) \in \mathbb{S}_t \cap g^{-1}(0).$ Clearly, there exists an injective analytic curve $\alpha \colon (-1, 1) \to \mathbb{R}^2$ lying in $\mathbb{S}_t$ such that $\alpha(0) = (\bar{x}, \bar{y}).$ We have proved that the function $s \mapsto \frac{f (\alpha(s))}{g (\alpha(s))}$ is constant, say $c,$ for $s > 0$ small enough. Since the function $s \mapsto {f (\alpha(s))} - c\, {g (\alpha(s))}$ is analytic, it follows that $s \mapsto \frac{f (\alpha(s))}{g (\alpha(s))}$ also is constant $c$ for $s < 0$ small. Therefore, the restriction of $\frac{f}{g}$ on $\mathbb{S}_t \setminus g^{-1}(0)$ is constant.

Recall that $m := \ord f$ and $n := \ord g.$
Let $(a,b)\in\mathbb R^2\setminus\{(0,0)\}$ such that $a^2+b^2=1$, $f_m(a,b)\ne 0$ and $g_n(a,b)\ne 0.$
We can write
\begin{eqnarray*}
f(at, bt) &=& f_m(a,b)\, t^m + \textrm{ higher order terms in } t, \\
g(at, bt) &=& g_n(a,b)\, t^n + \textrm{ higher order terms in } t,
\end{eqnarray*}
By assumption, $m > n.$ Therefore, as $t := \sqrt{x^2 + y^2} \to 0$ with $(x, y) \not \in g^{-1}(0),$ we have
$$\frac{f(x, y)}{g(x, y)} \ = \ \frac{f(at, bt)}{g(at, bt)} \ \to \ 0,$$
which completes the proof.
\end{proof}

We next consider the case $G_{f, g} \not \equiv 0.$ To this end, we need the following fact.

\begin{lemma}
Let $f$ and $g$ be $y$-regular of order $m$ and $n$, respectively. Denote by $f_m$ and $g_n$ the homogeneous components of $f$ and $g$ of the degrees $m$ and $n,$ respectively. Suppose that $m \neq n,$ then the following statements hold.
\begin{enumerate}[\upshape (i)]
\item If $\left(g_n\frac{\partial f_m}{\partial x} - f_m \frac{g_n}{\partial x}\right)(0,1)\neq 0,$ then $G_{f, g}$ is $y$-regular.
\item If $\left(g_n\frac{\partial f_m}{\partial x} - f_m\frac{g_n}{\partial x}\right)(0,1) =  0,$ then the polynomials 
$\widetilde{f}(x,y) := f(x, x+y),$ $\widetilde{g}(x,y) := g(x, x+y),$ as well as, $G_{\widetilde{f}, \widetilde{g}}$ are $y$-regular.
\end{enumerate}
\end{lemma}

\begin{proof}
(i): This is trivial.

(ii): The polynomials $f$ and $g$ are $y$-regular, so are the polynomials $\widetilde{f}$ and $\widetilde{g}.$ Furthermore, we have 
\[
G_{\widetilde{f}, \widetilde{g}}(x,y)=y\left[\left(g\frac{\partial f}{\partial x}-f\frac{\partial g}{\partial x}\right)+\left(g\frac{\partial f}{\partial y}-f\frac{\partial g}{\partial y}\right)\right](x,x+y)+x\left(g\frac{\partial f}{\partial y}-f\frac{\partial g}{\partial y}\right)(x,x+y).
\]
The coefficient of $y^{m+n}$ in $G_{\widetilde{f}, \widetilde{g}}$ is 
\begin{eqnarray*}
\left(g_n\frac{\partial f_m}{\partial x} - f_m\frac{\partial g_n}{\partial x}\right)(0,1) + \left(g_n\frac{\partial f_m}{\partial y}-f_m\frac{\partial g_n}{\partial y}\right) (0, 1) 
&=& \left(g_n\frac{\partial f_m}{\partial y}-f_m\frac{\partial g_n}{\partial y}\right)(0,1)\\
&=& (m - n) f_m(0, 1) g_n(0, 1)\neq 0,
\end{eqnarray*}
which implies that $G_{\widetilde{f}, \widetilde{g}}$ is $y$-regular of order $m + n.$
\end{proof}

\begin{theorem}\label{Theorem415}
Assume that $f, g$ and $G_{f,g}$ are $y$-regular. Then $\displaystyle\lim_{(x, y) \to (0, 0)}\frac{f(x, y)}{g(x, y)} = 0$ if and only if the following conditions hold:
\begin{enumerate}[{\rm (i)}]
\item $\ord f > \ord g.$
\item For $*\in\{\pm\}$ and  for any $\gamma\in \mathbf{P}^{*}(G_{f,g}) \setminus \big(\mathbf{P}^{*}(f) \cup \mathbf{P}^{*}(g) \big),$ we have
\begin{eqnarray*}
\ord{f(* x, \gamma(x))} &>& \ord{g(* x, \gamma(x))}.
\end{eqnarray*}

\item For $*\in\{\pm\}$ and for any $\gamma \in \mathbf{P}^{*}(g),$ we have $\gamma \in \mathbf{P}^{*}(f)$ and either $\mult_f(\gamma) > \mult_g(\gamma)$ or $\mult_f(\gamma) = \mult_g(\gamma)$ and
$$\sum_{\varphi\in \mathbf{P}(f) \setminus\{\gamma\}}\mult_f(\varphi)\
\ord(\gamma  - \varphi)>\sum_{\psi\in \mathbf{P}(g)\setminus\{\gamma\}}\mult_g(\psi)\ \ord(\gamma  - \psi).$$
\end{enumerate}
\end{theorem}

\begin{proof} 
First of all, assume that $\displaystyle\lim_{(x, y) \to (0, 0)}\frac{f(x, y)}{g(x, y)} = 0,$ then (i) and (ii) follow immediately. Furthermore, given $\epsilon > 0$ we can find $\delta > 0$ such that
\begin{eqnarray} \label{EqnLimit}
|f(x, y)| &\leqslant & \epsilon |g(x, y)| \quad \textrm{ for all } \quad (x, y) \in \mathbb{B}_\delta.
\end{eqnarray}
This yields $\mathbf{P}^{*}(g) \subset \mathbf{P}^{*}(f).$ Next, take any $\gamma \in \mathbf{P}^{+}(f).$ (The case where $\gamma \in \mathbf{P}^{-}(f)$ is proved similarly.)
By Puiseux's theorem, we can write
\begin{eqnarray*}
f(x, y) &=& f_1(x, y)  \big(y - \gamma(x) \big)^{\mult_f(\gamma)}, \\
g(x, y) &=& g_1(x, y)  \big(y - \gamma(x) \big)^{\mult_g(\gamma)}, 
\end{eqnarray*}
where $f_1, g_1 \in \mathbb{R}\{x^{\frac{1}{d}}\}[y]$ with $f_1(x, \gamma(x)) \ne 0$ and $g_1(x, \gamma(x)) \ne 0$ for $x > 0$ small enough and $d$ is a positive integer.
From \eqref{EqnLimit}, we deduce that either $\mult_f(\gamma) > \mult_g(\gamma)$ or $\mult_f(\gamma) =  \mult_g(\gamma).$ Furthermore, in the second case we have
\begin{eqnarray*}
\lim_{(x, y) \to (0^{+}, 0)}\frac{f_1(x, y)}{g_1(x, y)} &=& \lim_{(x, y) \to (0^{+}, 0)}\frac{f(x, y)}{g(x, y)} \ = \ 0,
\end{eqnarray*}
which yields $\mathrm{ord} f_1(x, \gamma(x)) > \mathrm{ord} g_1(x, \gamma(x)).$ Therefore (iii) holds.

Now let the conditions (i)--(iii) hold. Clearly, it suffices to show that
\begin{eqnarray*}
\lim_{(x, y) \to (0^{+}, 0)}\frac{f(x, y)}{g(x, y)} &=& 0 \quad \textrm{ and } \quad \lim_{(x, y) \to (0^{-}, 0)}\frac{f(x, y)}{g(x, y)} \ = \ 0.
\end{eqnarray*}
We prove only the first equality; the second equality can be proved similarly.

The condition~(iii) gives $\mathbf{P}^{+}(g) \subset \mathbf{P}^{+}(f).$ Hence, there exist $\delta > 0$, a positive integer $d$ and $f_1, g_1 \in \mathbb{R}\{x^{\frac{1}{d}}\}[y]$ with $g_1^{-1}(0) \cap \mathbb{B}_\delta \subset \{(0, 0)\}$ such that for all $(x, y) \in \mathbb{B}_\delta,$ we have
\begin{eqnarray*}
f(x, y) &=& f_1(x, y) \prod_{\varphi \in  \mathbf{P}^{+}(f)}  \big(y - \varphi(x) \big)^{\mult_f(\varphi)}, \\
g(x, y) &=& g_1(x, y)  \prod_{\varphi \in  \mathbf{P}^{+}(f)}  \big(y - \varphi(x) \big)^{\mult_g(\varphi)}.
\end{eqnarray*}
For simplicity, write $\mathbb{B}^{+}_\delta := \mathbb{B}_\delta \cap \{x \geqslant 0\}.$ By assumption, the semi-algebraic function $h \colon \mathbb{B}^{+}_\delta \setminus \{(0, 0)\} \to \mathbb{R}$ defined by
\begin{eqnarray*}
h(x, y) &=& \frac{f_1(x, y)}{g_1(x, y)} \prod_{\varphi \in  \mathbf{P}^{+}(f)}  \big(y - \varphi(x) \big)^{\mult_f(\varphi) -  \mult_g(\varphi)},
\end{eqnarray*}
is well-defined and continuous. Furthermore, a direct calculation shows that for all $t \in (0, \delta),$
\begin{eqnarray*}
\min_{(x, y) \in \mathbb{S}^{+}_t} h(x, y) &=& \inf_{(x, y) \in \mathbb{S}^{+}_t \setminus g^{-1}(0)} \frac{f(x, y)}{g(x, y)}, \\
\max_{(x, y) \in \mathbb{S}^{+}_t} h(x, y) & = &  \sup_{(x, y) \in \mathbb{S}^{+}_t \setminus g^{-1}(0)} \frac{f(x, y)}{g(x, y)},
\end{eqnarray*}
where $\mathbb{S}^{+}_t := \mathbb{S}_t \cap \{x \geqslant 0\}.$ Applying Lemma~\ref{LemmaSAChoice} for the (nonempty) semi-algebraic set
\begin{eqnarray*}
S & := &  \Big  \{(t, (x, y)) \in (0, \delta)  \times \mathbb{B}^{+}_\delta \ \Big | \  (x, y) \in \mathrm{argmin}_{(x', y') \in \mathbb{S}^{+}_t} h(x', y') \Big \}
\end{eqnarray*}
and the projection $S  \to (0, \delta), ({t}, (x, y)) \mapsto {t},$ we get a semi-algebraic curve 
$$\gamma \colon (0, \delta) \to \mathbb{R}^2, \quad t \mapsto ((x(t), y(t)),$$ 
such that for all $t \in (0, \delta)$ we have 
\begin{eqnarray*}
\gamma(t) \ \in \ \mathbb{S}^{+}_t \quad \textrm{ and } \quad  h(\gamma(t)) &=& \min_{(x, y) \in \mathbb{S}^{+}_t} h(x, y).
\end{eqnarray*}
We will show that $\lim_{t \to 0^+} h(\gamma(t)) = 0.$

Indeed, by Lemma~\ref{MonotonicityLemma} and by shrinking $\delta$ (if necessary), we may assume that one of the following conditions hold:
\begin{itemize}
\item $x(t) \equiv 0;$
\item $x(t)  > 0$ and $g(\gamma(t)) \equiv 0;$ 
\item $x(t)  > 0$ and $g(\gamma(t)) \ne  0.$
\end{itemize}
In the first case, it follows from the condition~(i) and the assumption that $f$ and $g$ are $y$-regular that $h(\gamma(t)) \to 0$ as $t \to 0.$ 
In the second case, the germ of the curve $t \mapsto \gamma(t)$ is a right half-branch of $g^{-1}(0),$ which, together with the condition~(iii), yields that $h(\gamma(t)) \to 0$ as $t \to 0.$  Finally, assume that the third case is satisfied. Then
\begin{eqnarray*}
h(\gamma(t)) &=& \frac{f(\gamma(t)) }{g(\gamma(t))},
\end{eqnarray*}
and so $\gamma(t)$ is an optimal solution of the problem $\inf_{(x, y)
	\in \mathbb{S}^{+}_t \setminus g^{-1}(0)} \frac{f(x, y)}{g(x, y)}
	.$ By the Lagrange multiplier theorem, the vectors $\nabla
	(\frac{f}{g})(\gamma(t))$ and $\gamma(t)$ are linearly dependent,
	i.e., $G_{f, g}(\gamma(t)) \equiv 0.$ 
Since $x(t) > 0$ for all $t \in (0, \delta),$ the germ of the curve $t \mapsto \gamma(t)$ is a right half-branch of $G_{f, g}^{-1}(0).$ This, together with the condition~(ii), yields that $h(\gamma(t)) \to 0$ as $t \to 0.$

Summarily, we always have 
\begin{eqnarray*}
\inf_{(x, y) \in \mathbb{S}^{+}_t  \setminus g^{-1}(0)} \frac{f(x, y)}{g(x, y)}  &=& \min_{(x, y) \in \mathbb{S}^{+}_t} h(x, y) \ = \ h(\gamma(t)) \ \to \ 0 \quad \textrm{ as } \quad t \to 0^+.
\end{eqnarray*}
Similarly, we also have 
\begin{eqnarray*}
\sup_{(x, y) \in \mathbb{S}^{+}_t  \setminus g^{-1}(0)} \frac{f(x, y)}{g(x, y)}  &=& \max_{(x, y) \in \mathbb{S}^{+}_t} h(x, y) \ \to \ 0 \quad \textrm{ as } \quad t \to 0^+.
\end{eqnarray*}
Therefore, $\displaystyle\lim_{(x, y) \to (0^{+}, 0)}\frac{f(x, y)}{g(x, y)} = 0,$ which ends the proof of the theorem.
\end{proof}

In what follows, if $\varphi$ is a Newton--Puiseux root of $f$ with multiplicity $\mult_f(\varphi),$ then we let
\begin{eqnarray*}
\mult_f(\varphi\mod x^N) &:=& \mult_f(\varphi).
\end{eqnarray*}
The following corollary is a truncated version of Theorem~\ref{Theorem415}. 

\begin{corollary} \label{Corollary416}
Under the assumptions of Theorem~\ref{Theorem415}, let  $N$ be an integer such that 
\begin{eqnarray*}
N &>& \max\{\mathcal N(\deg  f + \deg G_{f, g}), \mathcal N(\deg  g + \deg G_{f, g})\}.
\end{eqnarray*}
Then $\displaystyle\lim_{(x, y) \to (0, 0)}\frac{f(x, y)}{g(x, y)} = 0$ if and only if the following conditions are satisfied:
\begin{enumerate}[{\rm (i)}]
\item $\ord f > \ord g.$
\item For $*\in\{\pm\}$ and for any $\gamma_N \in \mathbf{P}_N^{*}(G_{f,g}) \setminus \big (\mathbf{P}_N^{*}(f) \cup \mathbf{P}_N^{*}(g) \big) ,$ we have
\begin{eqnarray*}
\ord{f(* x, \gamma_N(x))} &>& \ord{g(* x, \gamma_N(x))}.
\end{eqnarray*}

\item For $*\in\{\pm\}$ and for any $\gamma_N \in \mathbf{P}_N^{*}(g),$ we have $\gamma_N \in \mathbf{P}_N^{*}(f)$ and either $\mult_f(\gamma_N) > \mult_g(\gamma_N)$ or $\mult_f(\gamma_N) = \mult_g(\gamma_N)$ and
\begin{eqnarray*}
\sum_{\varphi_N \in \mathbf{P}_N(f)\setminus\{\gamma_N\}}\mult_f(\varphi_N)\ \ord(\gamma_N - \varphi_N) &>& \sum_{\psi_N\in \mathbf{P}_N(g)\setminus\{\gamma_N\}}\mult_g(\psi_N)\ \ord(\gamma_N - \psi_N).
\end{eqnarray*}
\end{enumerate}
\end{corollary}

\begin{proof}
It is enough to show that the conditions (ii)--(iii) are equivalent to the conditions (ii)--(iii) in Theorem~\ref{Theorem415}.

(ii) $\Leftrightarrow$ Theorem~\ref{Theorem415}(ii):
By Theorem~\ref{Theorem37} and the choice of $N,$ the truncation mapping $\gamma \mapsto \gamma \mod x^N$ induces 
a bijection between $\mathbf{P}^{*}(G_{f,g}) \setminus \mathbf{P}^{*}(f)$ and $\mathbf{P}_N^{*}(G_{f,g})\setminus \mathbf{P}_N^{*}(f)$ and a bijection between $\mathbf{P}^{*}(G_{f,g}) \setminus \mathbf{P}^{*}(g)$ and $\mathbf{P}_N^{*}(G_{f,g})\setminus \mathbf{P}_N^{*}(g).$
Observe that
\begin{eqnarray*}
\mathbf{P}^{*}(G_{f,g})\setminus \big (\mathbf{P}^{*}(f) \cup \mathbf{P}^{*}(g) \big ) &=& \big (\mathbf{P}^{*}(G_{f,g})\setminus \mathbf{P}^{*}(f) \big ) \cap \big (\mathbf{P}^{*}(G_{f,g})\setminus \mathbf{P}^{*}(g) \big ), \\
\mathbf{P}_N^{*}(G_{f,g})\setminus \big (\mathbf{P}_N^{*}(f) \cup
\mathbf{P}_N^{*}(g) \big ) &=&  \big (\mathbf{P}_N^{*}(G_{f,g})\setminus \mathbf{P}_N^{*}(f) \big ) \cap \big (\mathbf{P}_N^{*}(G_{f,g})\setminus \mathbf{P}_N^{*}(g) \big ).
\end{eqnarray*}
Therefore, the mapping $\gamma \mapsto \gamma \mod x^N$ also induces a bijection between $\mathbf{P}^{*} \big (G_{f,g})\setminus (\mathbf{P}^{*}(f) \cup \mathbf{P}^{*}(g) \big )$ and $\mathbf{P}_N^{*}(G_{f,g}) \setminus \big (\mathbf{P}_N^{*}(f)\cup \mathbf{P}_N^{*}(g) \big ).$
In addition, by Theorem~\ref{Theorem36}(iv) and the choice of $N,$ one has for all $\gamma \in \mathbf{P}^{*} (G_{f,g}) \setminus \big (\mathbf{P}^{*}(f) \cup \mathbf{P}^{*}(g) \big),$
\begin{eqnarray*}
\ord{f(* x, \gamma(x))} &=& \ord{f(* x, \gamma(x)\mod x^N)}, \\ 
\ord{g(* x, \gamma(x))} &=& \ord{g(* x, \gamma(x)\mod x^N)}.
\end{eqnarray*}
Hence (ii) and Theorem~\ref{Theorem415}(ii) are equivalent.

(iii) $\Leftrightarrow$ Theorem~\ref{Theorem415}(iii): This follows directly from Theorems~\ref{Theorem36} and \ref{Theorem37}. The details are left to the reader.
\end{proof}

\begin{remark}[compare \cite{Xiao2015}] \label{Remark417}{\rm
(i) By definition, each truncated Puiseux series $(x, \gamma_N(x))$ gives a polynomial mapping $\mathbb{R} \to \mathbb{R}^2, t \mapsto (at^d, \gamma_N(at^d))$ for some $a \in \mathbb{R}, a \ne 0,$ and $d \in \mathbb{N}.$ In view of Corollary~\ref{Corollary416}, for determining the existence of limits of rational functions, it suffices to study limits along some polynomial curves. 

(ii) By the curve selection lemma (see \cite{Milnor1968}), it is not hard to see that for two {\em multivariate} polynomials $f, g \in\mathbb R[z],$ the (finite) limit $\displaystyle \lim _{z \to 0} \frac{f(z)}{g(z)}$ does not exist if and only if one of the following conditions holds:
\begin{enumerate}[$\bullet$]
\item There exists a polynomial mapping $u \colon \mathbb R\to \mathbb R^n$ such that 
$$u(0) = 0,\ \ g(u(t))\not= 0  \text{ for } 0<|t|\ll 1\ \text{ and }\ \lim _{t \to 0} \frac{f(u(t))}{g(u(t))}=\infty .$$
\item There exist two polynomial mappings $u\colon\mathbb R\to \mathbb R^n$ and $w\colon\mathbb R\to \mathbb R^n$ such that 
$$u(0)=w(0)=0,\ \ g(u(t))g(w(t))\not= 0  \text{ for } 0<|t|\ll 1 \ \text{ and }\ \lim _{t \to 0} \frac{f(u(t))}{g(u(t))}\not=\lim _{t \to 0} \frac{f(w(t))}{g(w(t))} .$$
\end{enumerate}
This fact is just the main result in \cite[Theorem 3.3]{Xiao2015} with a different (long) proof.
}\end{remark}

\section{Ranges, an algorithm and numerical experiments}\label{Section5}
\subsection{Ranges}
For the case when the limit $\displaystyle\lim_{(x, y) \to (0, 0)}\frac{f(x, y)}{g(x, y)}$ does not exist, we still can compute the {\em (numerical) range} of limits, defined to be the following set
$$\mathcal{R}_{f, g} :=\left\{
L\in \overline{\mathbb{R}} : \ \text{there is a sequence } (x_k,y_k)\to(0,0) \ \text{ s.t. } \displaystyle\lim_{(x_k, y_k) \to (0, 0)}\frac{f(x_k, y_k)}{g(x_k, y_k)}=L
\right\}.$$
By Theorem~\ref{TarskiSeidenbergTheorem}, it is not hard to see that $\mathcal{R}_{f, g}$ is a semi-algebraic set and so it is a finite union of points and intervals. In general, the range $\mathcal{R}_{f, g}$ may be not connected. However, as shown below, if  the denominator $g$ has an isolated zero at the origin, then $\mathcal{R}_{f, g}$ is a closed interval. 

\begin{theorem}\label{Proposition410}
Assume that $g$ has an isolated zero at the origin and the polynomials $f, g$ and $F_{f, g}$ are $y$-regular. Let 
$$\mathrm{MIN}:=\min_{\substack{*\in\{\pm\}\\ \gamma\in \mathbf{P}^{*}(F_{f, g})}} \lim_{x\to 0^+}\frac{f(* x, \gamma(x))}{g(* x, \gamma(x))} \in \overline{\mathbb R}$$
and
$$\mathrm{MAX}:=\max_{\substack{*\in\{\pm\}\\ \gamma\in \mathbf{P}^{*}(F_{f, g})}} \lim_{x\to 0^+}\frac{f(* x, \gamma(x))}{g(* x, \gamma(x))} \in \overline{\mathbb R}.$$
Then $$\mathcal{R}_{f, g} = [\mathrm{MIN},\mathrm{MAX}].$$
In particular, if $\mathrm{MIN}=\mathrm{MAX} \in \mathbb{R},$ then the limit $\displaystyle\lim_{(x, y) \to (0, 0)}\frac{f(x, y)}{g(x, y)}$ exists and is equal to this number.
\end{theorem}

\begin{proof}
The last statement is clear from the first one so it is enough to prove the first one.
Similarly to the proof of Theorem~\ref{Theorem48}, there are positive constants $r, \delta$ and two semi-algebraic curves 
$$\alpha,\beta\colon(0,\delta)\to \mathbb R^2$$
with $F_{f, g}(\alpha(t)) \equiv 0$ and $F_{f, g}(\beta(t)) \equiv 0$ such that for all $t \in (0, \delta)$ we have that $\alpha(t), \beta(t) \in g^{-1}(t) \cap   \mathbb{B}_r$ are respectively optimal solutions of the problems $\min_{(x, y) \in g^{-1}(t) \cap \mathbb{B}_r} f(x, y)$ and $\max_{(x, y) \in g^{-1}(t) \cap \mathbb{B}_r} f(x, y)$.
Clearly 
\begin{equation}\label{MINMAX}
\mathrm{MIN}=\lim_{t \to 0^+}\frac{f(\alpha(t))}{g(\alpha(t))} \quad \text{ and } \quad \mathrm{MAX}=\lim_{t \to 0^+}\frac{f(\beta(t))}{g(\beta(t))}.
\end{equation}

Let $(x_k, y_k) \in \mathbb{R}^2 \setminus g^{-1}(0)$ be a sequence tending to $(0, 0)$ with $\displaystyle\lim_{(x_k, y_k) \to (0, 0)}\frac{f(x_k, y_k)}{g(x_k, y_k)}=L.$ 
Then for all sufficiently large $k,$ we have $(x_k, y_k) \in \mathbb{B}_r$ and $t_k := g(x_k, y_k) \in (0, \delta).$ 
For such $k$, the following inequalities hold:
\begin{eqnarray*}
 \frac{f(\alpha(t_k))}{g(\alpha(t_k))}\ = \min_{(x, y) \in g^{-1}(t_k) \cap \mathbb{B}_r} \frac{f(x, y)}{g(x, y)}  &\leqslant& 
\frac{f (x_k, y_k) }{g (x_k, y_k) } \ \leqslant \ \max_{(x, y) \in g^{-1}(t_k) \cap \mathbb{B}_r} \frac{f(x, y)}{g(x, y)}\ = \ \frac{f(\beta(t_k))}{g(\beta(t_k))}.
\end{eqnarray*}
Letting $k\to\infty,$ we get $\mathrm{MIN}\leqslant L\leqslant\mathrm{MAX}.$
Thus $\mathcal{R}_{f, g} \subset [\mathrm{MIN}, \mathrm{MAX}].$

On the other hand, note that $\mathrm{MIN},\mathrm{MAX}\in \mathcal{R}_{f, g}.$ Let $\widetilde L\in (\mathrm{MIN},\mathrm{MAX}).$
Let us prove $\widetilde L\in \mathcal{R}_{f, g}.$ Without loss of generality, assume that $g \geqslant 0.$
By Lemma~\ref{Lemma43} and by shrinking $r$ and $\delta$ if necessary, for each $t$ in $(0, \delta)$, the set $g^{-1}(t)\cap \mathbb B_r$ is a connected closed curve. Hence, by continuity, $f(g^{-1}(t)\cap \mathbb B_r)$ and so $\frac{f}{g}(g^{-1}(t)\cap \mathbb B_r)$ are closed intervals.
Moreover, for all $t>0$ small enough, by~\eqref{MINMAX}, we have 
$$\frac{f(\alpha(t))}{g(\alpha(t))}\leqslant\widetilde L\leqslant \frac{f(\beta(t))}{g(\beta(t))}.$$
Thus $\widetilde L\in \frac{f}{g}(g^{-1}(t)\cap \mathbb B_r).$
This implies  $\widetilde L\in \mathcal{R}_{f, g}$ and so $[\mathrm{MIN},\mathrm{MAX}]\subset \mathcal{R}_{f, g}.$
Consequently, $\mathcal{R}_{f, g} = [\mathrm{MIN},\mathrm{MAX}].$
\end{proof}

\begin{corollary}
Under the assumptions of Theorem~\ref{Proposition410}. Let $N$ be an integer such that
\begin{eqnarray*}
N &>& \max\{\mathcal N(\deg  f+\deg F_{f,g}), \mathcal N(\deg  g)\}.
\end{eqnarray*}
We have
$$\mathrm{MIN}=\left\{
\begin{array}{llll}
\displaystyle \min_{\substack{*\in\{\pm\}\\ \gamma_N\in \mathbf{P}^{*}_N(F_{f, g})}} \lim_{x\to 0^+}\frac{f(* x, \gamma_N(x))}{g(* x, \gamma_N(x))} & \text{ if }\ \mathbf{P}^{+}_N(f)\cup \mathbf{P}^{-}_N(f)=\emptyset\\
\displaystyle \min\left\{0,\min_{\substack{*\in\{\pm\}\\ \gamma_N\in \mathbf{P}^{*}_N(F_{f, g})\setminus\mathbf{P}^{*}_N(f)}} \lim_{x\to 0^+}\frac{f(* x, \gamma_N(x))}{g(* x, \gamma_N(x))}\right\} & \text{ otherwise }
\end{array}
\right.$$
and
$$\mathrm{MAX}=\left\{
\begin{array}{llll}
\displaystyle \max_{\substack{*\in\{\pm\}\\ \gamma_N\in \mathbf{P}^{*}_N(F_{f, g})}} \lim_{x\to 0^+}\frac{f(* x, \gamma_N(x))}{g(* x, \gamma_N(x))} & \text{ if }\ \mathbf{P}^{+}_N(f)\cup \mathbf{P}^{-}_N(f)=\emptyset\\
\displaystyle \max\left\{0,\max_{\substack{*\in\{\pm\}\\ \gamma_N\in \mathbf{P}^{*}_N(F_{f, g})\setminus\mathbf{P}^{*}_N(f)}} \lim_{x\to 0^+}\frac{f(* x, \gamma_N(x))}{g(* x, \gamma_N(x))}\right\} & \text{ otherwise. }
\end{array}
\right.$$

\end{corollary}
\begin{proof}
In view of Theorems~\ref{Theorem36}(iv), by remaking that $\mathbf{P}^{*}(g)=\emptyset$, for $*\in\{\pm\}$ and $\gamma\in \mathbf{P}^{*}(F_{f, g})\setminus\mathbf{P}^{*}(f),$ we have
$$\nu(f(* x, \gamma(x)))=\nu(f(* x, \gamma_N(x))) \ \text{ and }\ \nu(g(* x, \gamma(x)))=\nu(g(* x, \gamma_N(x))).$$
Hence
$$\lim_{x\to 0^+}\frac{f(* x, \gamma(x))}{g(* x, \gamma(x))}=\lim_{x\to 0^+}\frac{f(* x, \gamma_N(x))}{g(* x, \gamma_N(x))}.$$
Furthermore, by the choice of $N$ and by Theorems~\ref{Theorem36}(ii)--(iii) and~\ref{Theorem37}, for $*\in\{\pm\}$, the truncation mapping $\varphi \mapsto \varphi_N:=\varphi \mod x^N$ induced a bijection between $\mathbf{P}^{*}(F_{f,g})$ and $\mathbf{P}_N^{*}(F_{f,g})$ and a bijection between $\mathbf{P}^{*}(F_{f, g})\setminus\mathbf{P}^{*}(f)$ and $\mathbf{P}_N^{*}(F_{f, g})\setminus\mathbf{P}_N^{*}(f)$. 
Consequently, in light of Theorem~\ref{Proposition410}, the corollary follows.
\end{proof}

We now present an algorithm for determining the limits/ranges of bivariate rational functions, which can be implemented  in a computer algebra system\footnote{The code for Maple, together with a Maple worksheet, is available at the following links: \url{https://drive.google.com/file/d/1Vz6-LLERLG1IaIPn8FBpbQ-WHFvIeWru/view} and \url{https://pan.baidu.com/s/1lG8mbjHs3S7Jgrys0KNPNw?pwd=ex9q}.}.


\subsection*{Algorithm BiLimit} \

INPUT: Two polynomials $f$ and $g$ in $\mathbb{Q}[x,y]$ of positive orders. 

OUTPUT: Decide whether or not the limit $\displaystyle\lim_{(x, y) \to (0, 0)}\frac{f(x, y)}{g(x, y)}$ exists and compute the limit/range.

\begin{enumerate}[{\rm Step 1.}]


\item Compute $p \in \mathbb{Q}[x, y],$ which is the greatest common
	divisor of the polynomials $f$ and $g.$ Replace $f$ and $g$ by the
	new polynomials $\frac{f}{p}$ and $\frac{g}{p},$ respectively. If
	$g(0, 0) \ne 0,$ then the limit $\displaystyle\lim_{(x, y) \to (0, 0)}\frac{f(x, y)}{g(x, y)}$ exists and equals to $\displaystyle\frac{f(0,0)}{g(0,0)}$  and the algorithm stops. Otherwise, proceed to the next step.


\item Set $\displaystyle F_{f,g}:= \frac{\partial f}{\partial x}\frac{\partial g}{\partial y}-\frac{\partial f}{\partial y}\frac{\partial g}{\partial x}$. 
If one of the polynomials $f, g$ and $F_{f, g}$ is not $y$-regular, make a linear transformation, so that the new polynomials $f, g,$ and $F_{f,g}$ are $y$-regular.


\item Let $M$ be the smallest integer\footnote{Maple can determine the best order of truncation to distinguish Newton--Puiseux roots of $g$ by using the command \textsf{puiseux}$(g,x=0,y,0)$.}
such that, for $*\in\{\pm\}$, the truncation mapping $\varphi \mapsto \varphi \mod x^M$ induces a bijection between 
$\mathbf{P}^{*} \left(g\right)$ and $\mathbf{P}_M^{*} \left(g\right)$
and compute $\mathbf{P}^{*}_M({g}).$ 
If $\mathbf{P}^{+}_M({g})$ or $\mathbf{P}^{-}_M({g})$ is nonempty, then the limit $\displaystyle\lim_{(x, y) \to (0, 0)}\frac{f(x, y)}{g(x, y)}$ does not exist and the algorithm stops. (The range is not computed in this case.)
Otherwise, proceed to the next step.

\item If $F_{f, g} \equiv 0,$ perform the following commands:
\begin{enumerate}[\qquad $\bullet$]
\item compute $\displaystyle L:=\lim_{y\to 0^{+}}\frac{f(0,y)}{g(0,y)}\in \overline{\mathbb R};$
\item if $L\in\mathbb R$, then the limit $\displaystyle\lim_{(x, y) \to (0, 0)}\frac{f(x, y)}{g(x, y)}$ exists and equals to $L;$
else the limit $\displaystyle\lim_{(x, y) \to (0, 0)}\frac{f(x, y)}{g(x, y)}$ does not exist and return the range $\{L\};$ 

\item the algorithm stops.
\end{enumerate}
Otherwise, i.e., $F_{f, g} \not\equiv 0,$  proceed to the next step.

\item Compute $\mathbf{P}^{*}_N({F}_{f, g})$ and $\mathbf{P}^{*}_N(f)$ where $N\geqslant M$ is the smallest integer 
 such that, for $*\in\{\pm\}$, the truncation mapping $\varphi \mapsto \varphi \mod x^N$ induces bijections between the following pairs of sets:
\begin{enumerate}[\qquad $\bullet$]
\item $\mathbf{P}^{*} \left(F_{f,g}\right)$ and $\mathbf{P}_N^{*} \left(F_{f,g}\right)$; 
\item $\mathbf{P}^{*} (f)$ and $\mathbf{P}_N^{*} (f)$; and 
\item $\mathbf{P}^{*} (F_{f, g})\setminus\mathbf{P}^{*}(f)$ and $\mathbf{P}_N^{*}(F_{f, g})\setminus\mathbf{P}_N^{*}(f)$.
\end{enumerate}

\item If $\mathbf{P}^{+}_N(f)\cup \mathbf{P}^{-}_N(f)=\emptyset,$ set $\textrm{MIN} :=+\infty$ and $\textrm{MAX} :=-\infty.$ Otherwise set $\mathrm{MIN}:=\mathrm{MAX}:=0.$
\item For each $\gamma_N\in \mathbf{P}_N^{*}\left(F_{f,g}\right)\setminus\mathbf{P}_N^{*}(f),$ let
$$\displaystyle L_{\gamma_N}:=\lim_{x\to 0^{+}}\frac{f(x,\gamma_N(x))}{g(x,\gamma_N(x))}\in \overline{\mathbb R}.$$ 

If $\text{MIN} > L_{\gamma_N}$, set $\text{MIN}:= {\gamma_N}.$ If $\text{MAX} < L_{\gamma_N},$ set $\text{MAX}:={\gamma_N}$. 

\item If $\text{MIN}=\text{MAX}\in\mathbb R$, then the limit $\displaystyle\lim_{(x, y) \to (0, 0)}\frac{f(x, y)}{g(x, y)}$ exists and equals to $\text{MIN}=\text{MAX}.$
	Otherwise, the limit $\displaystyle\lim_{(x, y) \to (0, 0)}\frac{f(x, y)}{g(x, y)}$
	does not exist and return the  range $[\text{MIN}, \text{MAX}].$
\end{enumerate}

\subsection*{Empirical results} \
{\upshape We implement our algorithm BiLimit in
		{\scshape Maple} 2021 and present an experimental comparison
		of our method with  Maple's built-in command {\sf limit/multi} for
		computing limits of the following bivariate rational
		functions at $(0,0)$. When the limit does not exist, the range can be also computed for the case of isolated zero at the origin of the denominator. 
		The rational functions in (1)-(2) are taken
		from \cite{Cadavid2013} and in (3)-(8) are taken from
		\cite{Zeng2020}. 
	\begin{enumerate}
		\item $f_1=x^4+x^2y+y^2$ and 
			  $g_1=x^2+y^2$. The limit does not exist. The range is $[0,1].$	
		\item $f_2=x^4+3x^2y-x^2-y^2$ and 
			  $g_2=x^2+y^2$. The limit is $-1$.
		\item $f_3=2y^5+x^2y^2-8xy^3-13y^4-2x^3+6x^2y+28xy^2+24y^3-4x^2-12xy-9y^2$ and 
			  $g_3=y^4-5xy^2-4y^3+7x^2+10xy+4y^2$. The limit does not exist. The range is $\left[-\frac{19}{3},0\right].$	
		\item $f_4=4x^2y^2-4xy^3+y^4-2xy^2+y^3$ and
			  $g_4=8x^2y^2-8xy^3+3y^4+8x^2-8xy+2y^2$. The limit does not exist. The range is $\left[-\frac{\sqrt{2}}{4},\frac{\sqrt{2}}{4}\right].$
		\item $f_5=10x^2y^2+x^3+2x^2y+4xy^2+6x^2+6xy+3y^2$ and
			  $g_5=3x^2y^2+2x^3+2xy+y^2$. The limit does not exist
			  and $(0,0)$ is a non-isolated zero of
		  $g_5$.
		\item $f_6=2x^2y^2+x^2y+2xy^2+y^3+x^2+2xy+2y^2$ and
			  $g_6=x^2y^2+x^2+2xy+2y^2$. The limit is $1$.
		\item $f_7=10x^2y^2+x^3+2x^2y+4xy^2+6x^2+6xy+3y^2$ and 
			  $g_7=3x^2y^2+2x^2+2xy+y^2$. The limit is $3$.
		\item $f_8=10x^2y-26x^3+37x^2y^2-8xy^3+2y^5-18xy^4+3y^6$ and
			  $g_8=24x^2 +3y^2-21xy^2-5xy+2y^3 + 5y^4$. The limit is $0$.
		\item $f_9=2x^2y^2+x^2y+2xy^2+y^3+x^2+2xy+2y^2$ and
			  $g_9=x^4y^4+x^2+2xy+2y^2$. The limit is $1$.
		\item $f_{10}=f_2f_7$ and 
			  $g_{10}=g_2g_7$. The limit is $-3$.
		\item $f_{11}=f_6f_7$ and 
			  $g_{11}=g_6g_7$. The limit is $3$.
		\item $f_{12}=f_2f_6$ and 
			  $g_{12}=g_2g_6$. The limit is $-1$.
		\item $f_{13}=f_2f_6+x^6y^6$ and 
			  $g_{13}=g_2g_6+x^4y^4$. The limit is $-1$.
		\item $f_{14}=f_2f_6+x^{10}y^{10}$ and 
			  $g_{14}=g_2g_6+x^8y^8$. The limit is $-1$.		
		\item $f_{15}=f_2f_6f_7$ and 
			  $g_{15}=g_2g_6g_7$. The limit is $-3$.
		\item $f_{16}=f_2f_6f_7f_8$ and 
			  $g_{16}=g_2g_6g_7g_8$. The limit is $0$.			
		\item $f_{17}=f_2f_6f_7f_9$ and 
			  $g_{17}=g_2g_6g_7g_9$. The limit is $-3$.		
		\item $f_{18}=x^2$ and 
			  $g_{18}=x^4+y^4$. The limit does not exist. The range is $[0,+\infty].$	
		\item $f_{19}=x^3$ and 
			  $g_{19}=x^4+y^4$. The limit does not exist. The range is $[-\infty,+\infty].$	
		\item $f_{20}=x^4+x^2y+y^2$ and 
			  $g_{20}=x^6+y^2$. The limit does not exist. The range is $\left[\frac{3}{4},+\infty\right].$	
		\item $f_{21}=x^4+x^2y^2+y^4$ and 
			  $g_{21}=x^6+y^4$. The limit does not exist. The range is $[1,+\infty].$
	\end{enumerate}

	Below we show the CPU time consumed by BiLimit and {\sf
	limit/multi} to compute the limit/range of $\frac{f_i}{g_i}$ at
	$(0,0)$ for $i=1,\dots,21$. The symbol ``$\times\times\times$''
	means that the limit does not exist but the range can not be
determined. 

$$\begin{array}{|c|c|c|c|c|}
\hline
i&\mathrm{BiLimit:time}&\mathrm{BiLimit:limit/range}&\mathsf{limit/multi:time}&\mathsf{limit/multi:limit/range}\\
\hline
1&0.375s&[0,1]&0.109s&[0,1]\\
\hline
2&0.156s&-1&0.125s&-1\\
\hline
3&0.281s&\left[-\frac{19}{3},0\right]&0.125s&\left[-\frac{19}{3},0\right]\\
\hline
4&0.203s&\left[-\frac{\sqrt{2}}{4},\frac{\sqrt{2}}{4}\right]&0.234s&\left[-\frac{\sqrt{2}}{4},\frac{\sqrt{2}}{4}\right]\\
\hline
5&0.047s&\times\times\times&0.062s&\times\times\times\\
\hline
6&0.187s&1&0.250s&1\\
\hline
7&0.187s&3&0.390&3\\
\hline
8&0.219s&0&0.421s&0\\
\hline
9&0.187s&1&0.531s&1\\
\hline
10&0.297s&-3&0.982s&-3\\
\hline
11&0.296s&3&1.358s&3\\
\hline
12&0.297s&-1&0.639s&-1\\
\hline
13&0.344s&-1&2.542s&-1\\
\hline
14&0.328s&-1&7.847s&-1\\
\hline
15&0.530s&-3&4.477s&-3\\
\hline
16&0.734s&0&0.874s-40.389s&0\\
\hline
17&3.151s&-3&34.554s&-3\\
\hline
18&0.140s&[0,+\infty]&0.172s&\times\times\times\\
\hline
19&0.140s&[-\infty,+\infty]&0.171s&\times\times\times\\
\hline
20&0.374s&\left[\frac{3}{4},+\infty\right]&0.156s&\left[\frac{3}{4},1\right]\ (\text{wrong})\\
\hline
21&0.406s&[1,+\infty]&0.125s&\times\times\times\\
\hline
\end{array}$$

Observe that both of our algorithm BiLimit and the Maple command {\sf limit/multi} can give the right answers for all these rational functions.
	For $\frac{f_{18}}{g_{18}}$, $\frac{f_{19}}{g_{19}}$ and $\frac{f_{21}}{g_{21}},$ our algorithm BiLimit can also compute the range while the Maple command {\sf limit/multi} cannot. 
	Furthermore, for $\frac{f_{20}}{g_{20}}$, our algorithm BiLimit returns the right range while the Maple command {\sf
	limit/multi} returns a wrong one.	
	Note also that our algorithm BiLimit is faster than the Maple command {\sf 	limit/multi}, except for $i=1,2,3,20,21$ for which all of them used very little time.
	In addition, it seems that the CPU time for the command \text{{\sf limit/multi}} to compute the limit of $\frac{f_{16}}{g_{16}}$ at $(0,0)$ is very inconsistent, i.e., the CPU time varies very much when computing repeatedly this example; so we list the less and most time spent in this example.
	The numerical experiments were carried out on a PC with dual
	64-bit Intel Xeon E5-2630 2.20 GHz CPUs and 32G RAM. 



\end{document}